\documentclass{amsart}

\usepackage[a4paper,centering,textwidth=145mm,textheight=195mm]{geometry}
\usepackage[all]{xypic}
\usepackage{amssymb}
\usepackage{color}

\newcommand{\bil}[1]{{\langle{#1}\rangle}} 
\newcommand{\bpf}[1]{{\langle\!\langle{#1}\rangle\!\rangle}} 
\newcommand{\qpf}[1]{{\langle\!\langle{#1}]]}} 

\newcommand{\Z}{\mathbb{Z}}

\DeclareMathOperator{\Symd}{Symd}
\DeclareMathOperator{\Sym}{Sym}
\DeclareMathOperator{\Trp}{Trp}
\DeclareMathOperator{\Srp}{Srp}
\DeclareMathOperator{\Nrp}{Nrp}
\DeclareMathOperator{\Prp}{Prp}
\DeclareMathOperator{\Nrd}{Nrd}
\DeclareMathOperator{\End}{End}

\DeclareMathOperator{\Int}{Int}
\DeclareMathOperator{\Br}{Br}

\DeclareMathOperator{\dlog}{dlog}

\DeclareMathOperator{\sq}{\mathsf{sq}}

\newtheorem{prop}{Proposition}[section]
\newtheorem{lemma}[prop]{Lemma}
\newtheorem{thm}[prop]{Theorem}

\newtheorem*{lemma*}{Lemma}

\theoremstyle{definition}
\newtheorem{defn}[prop]{Definition}

\theoremstyle{remark}
\newtheorem*{remark*}{Remark}
\newtheorem{remark}[prop]{Remark}

\title[Descent and decomposability of algebras of exponent $2$]{Descent and decomposability of algebras of exponent $2$ in characteristic $2$}
\author{Anne Qu\'eguiner-Mathieu}
\address{Universit\'e Sorbonne Paris Nord\\
Institut Galilée \\
LAGA - CNRS (UMR 7539)\\
F-93430 Villetaneuse, France}
\email{queguin@math.univ-paris13.fr}

\author{Jean-Pierre Tignol}
\address{ICTEAM Institute, Box L4.05.01\\
UCLouvain\\
B-1348 Louvain-la-Neuve, Belgium}
\email{jean-pierre.tignol@uclouvain.be}

\date{\today}



\thanks{The authors acknowledge
    support from Wallonie--Bruxelles International and the French
    government in the framework of ``Partenariats Hubert Curien''
    (Project \emph{$2$-alg\`ebres de degr\'e~$8$ et
      d'exposant~$2$}). Research leading to this work was initiated
    while the second author was visiting the Sorbonne Paris Nord
    university, whose hospitality is gratefully acknowledged.}

\begin{document}

 \begin{abstract}We define an invariant $\Delta$ of central simple algebras of degree 8 and exponent 2 over fields of characteristic 2. It is built from the discriminant Pfister forms of symplectic involutions, takes values in a quotient of the third part of the graded group associated to the fundamental filtration of the Witt group of quadratic forms, and vanishes on tensor products of three quaternion algebras, thus giving a necessary condition for decomposability. Variants of the construction yield necessary and sufficient conditions for a biquaternion algebra to descend from a quadratic extension of the base field, and for a quaternion algebra to descend from a biquadratic extension that is not purely inseparable. Our results extend known results both in characteristic $2$ and in characteristic different from $2$, see the introduction. Using residue maps for quadratic forms, we show that these invariants do not vanish in general, and we give an explicit central simple algebra of degree 8 and exponent 2 with nonzero $\Delta$-invariant, which is therefore indecomposable. As explained in the appendix, the purely inseparable case is quite different; in particular, a quaternion algebra descends if and only if its Brauer class also does.
    \end{abstract}
    
    \maketitle

Using a cohomological invariant of symplectic involutions,
Garibaldi--Parimala--Tignol define in~\cite[\S11]{GPT} a
cohomological invariant $\Delta$ of central simple algebras of
degree~$8$ and exponent~$2$ over arbitrary fields of characteristic
different from~$2$. This invariant vanishes when the algebra decomposes into a tensor product of
three quaternion algebras. Examples where it does not
vanish were given by Barry in~\cite{Barry};
Merkurjev--Neshitov--Zainoulline showed that it is nontrivial generically~\cite[\S4a]{MNZ}. The purpose of
the present paper is to describe a similar construction in
characteristic~$2$, based on the discriminant Pfister form of
symplectic involutions from~\cite{BGBT3} or \cite{CRAS}. Being rooted
in quadratic forms, the invariant $\Delta$ in
Definition~\ref{def:Delta} below takes values in a quotient of the
group $I^3_q$ from the fundamental filtration of the Witt group of the
base field. Kato's isomorphisms from~\cite{Kato} allow an
equivalent formulation in terms of Kato--Milne cohomology.

While the invariant $\Delta$ yields a necessary criterion for
decomposability of central simple algebras of degree~$8$ and
exponent~$2$, variants of the construction using discriminant Pfister
forms yield necessary and sufficient criteria for the descent of
biquaternion algebras over quadratic extensions, and for the descent
of quaternion algebras over biquadratic extensions that are not purely
inseparable. These variants are an additional theme of the paper; they
play a central role in the construction of examples showing that the
invariant $\Delta$ is not trivial.

The definition of the invariant $\Delta$ is given
in Section~\ref{sec:Delta}, see Definition~\ref{def:Delta}. It relies on the
proof that the discriminant 
$3$-fold Pfister forms of all the symplectic involutions on a central
simple algebra of degree~$8$ have the same image in a suitable
quotient of the group $I^3_q$ of the center. 

Section~\ref{sec:quadesc} discusses a necessary and sufficient
condition for a $2$-torsion Brauer class over a quadratic extension to
descend to a $2$-torsion Brauer class over the base field, i.e., to be
obtained by scalar extension from a $2$-torsion Brauer class over the
base field. This condition is expressed by the vanishing of the
corestriction when the extension is separable, and of a canonical map
induced by the squaring map when it is not, see
Theorem~\ref{thm:quadesc}. The result is well-known when the 
characteristic is different from~$2$. For Brauer classes of quaternion
algebras, a similar condition is given for the descent from a
biquadratic extension that is not purely inseparable, see
Theorem~\ref{thm:biquaddescBr}. 

Even when the Brauer class of a biquaternion algebra over a quadratic
extension descends, the algebra itself may not descend: it is Brauer-equivalent to the scalar extension of a central simple algebra defined over the base field, but this algebra may have larger degree. 
Given a quadratic extension $K$ of a base field $F$ of characteristic $2$, section~\ref{sec:Quadbiquat} introduces an
invariant $\delta_{K/F}$ of biquaternion algebras over $K$, whose vanishing characterizes biquaternion
algebras that descend, i.e., which are isomorphic to the scalar extension of a biquaternion algebra over $F$, see Definition~\ref{def:deltaKF} and Proposition~\ref{prop:deltaKFdesc}. The definition of $\delta_{K/F}$ follows the same pattern as
for the $\Delta$-invariant: it relies on the result that the given
biquaternion algebra $B$ over $K$ is contained in a central simple
$F$-algebra $A$ of degree~$8$, and that all the discriminant Pfister
forms of symplectic involutions on $A$ that restrict to a symplectic
involution on $B$ have the same image in a suitable quotient of
$I^3_q(F)$, independently of the choice of the algebra
$A$, see Definition~\ref{def:deltaKF}. This invariant turns out to be
the same as the one defined in cohomological terms by 
Barry--Chapman--Laghribi~\cite{BaCL} when $K$ is separable over $F$,
and by Laghribi--Lorenz~\cite[\S6]{LL} when $K$ is inseparable over
$F$ (and by Barry~\cite{Barry} when the characteristic is not~$2$), 
although the earlier constructions did not rely on the choice of an algebra $A$ containing $B$.

A similar strategy is used in Section~\ref{sec:biquadesc} to
define an invariant $\delta_{L/F}$ of quaternion algebras over a
biquadratic extension $L$ of a base field $F$ of characteristic~$2$,
assuming $L$ is not purely inseparable over $F$: when the Brauer class
of the quaternion algebra $Q$ descends from~$L$ to $F$, the algebra
$Q$ can be embedded in a central simple $F$-algebra $A$ of degree~$8$,
and all the discriminant Pfister forms of involutions on $A$ that
extend the canonical conjugation involution on $Q$ have the same image
in a suitable quotient of $I^3_q(F)$, independently of the
choice of $A$. This common image defines an invariant $\delta_{L/F}(Q)$, which
vanishes if and only if $Q$ descends to a quaternion algebra over
$F$, see Definition~\ref{def:deltaLFQ} and
Theorem~\ref{thm:biquadesc}. With a completely different method, this
invariant has also been 
defined by Laghribi--Lorenz~\cite{LL} (and by Sivatski~\cite{Siv} when
the characteristic is not~$2$), although these authors consider it as an
invariant of the central simple $F$-algebras that contain~$Q$. 
The purely inseparable case is considered in the appendix. The situation there is quite different: a quaternion algebra over a purely inseparable biquadratic field extension descends if and only if its Brauer class descends, see Theorem~\ref{thm:pureinsep}. 

As this rough sketch makes clear, the various
invariants defined in this paper are closely related: if a central simple $F$-algebra $A$ of
degree~$8$ and exponent~$2$ contains a biquadratic field extension $L$
of $F$ that is not purely inseparable, and $K\subset L$ is a quadratic
extension of $F$, then the centralizer of $L$ in $A$ is a quaternion
$L$-algebra $Q$ and the centralizer of $K$ in $A$ is a biquaternion
$K$-algebra $B$. The following inclusions are clear from the
definitions:
\[
  F\subset K\subset L\subset Q \subset B \subset A.
\]
Canonical homomorphisms induced by the inclusions carry the invariants
as follows:
\[
  \delta_{L/F}(Q) \mapsto \delta_{K/F}(B) \mapsto \Delta(A).
\]
Therefore, examples where $\Delta(A)\neq0$ yield examples of
biquaternion algebras $B$ such that $\delta_{K/F}(B)\neq0$ and of
quaternion algebras $Q$ with $\delta_{L/F}(Q)\neq0$. We assume
throughout that the characteristic is~$2$, but the reader may 
verify that all the results in
Sections~\ref{sec:Delta}--\ref{sec:biquadesc} are valid in arbitrary
characteristic, with only minor alterations in the proofs involving
separable field extensions.

Section~\ref{sec:nontriv} exhibits examples where the invariants do
not vanish, using residue maps for quadratic forms as the main tool. A
first construction takes as a starting ingredient any central simple
$F$-algebra $A$ of degree~$8$ and exponent~$2$ that does not decompose
into a tensor product of quaternion algebras. The
$\delta_{K/F}$-invariant of the centralizer of any separable quadratic
extension $K$ of $F$ in $A$ is shown to be nontrivial, even though it
is not known whether $\Delta(A)$ vanishes or not. A central simple
algebra of degree~$8$ and exponent~$2$ with nontrivial
$\Delta$-invariant over the rational function field $F(t)$ in one
indeterminate is built from $A$ and the quadratic extension $K$, see
Theorem~\ref{thm:const1}. 

A second construction is more explicit: it centers on the quaternion
algebra $[a_1a_2,a_1+a_2)_F$ over the rational function field
$F=k(a_1,a_2)$ in two indeterminates over an arbitrary field $k$ of
characteristic~$2$, and on the quadratic extension
$K=F(\sqrt{a_1a_2})$ and the biquadratic extension
$L=K(\wp^{-1}(a_2))$. With these ingredients, and a sequence of
rational function field extensions $F\subset F(t_1)\subset F(t_1,t_2)
\subset F(t_1,t_2,t_3)$, the following are successively constructed:
\begin{itemize}
\item
  a quaternion $L(t_1)$-algebra $Q$ such that
  $\delta_{L(t_1)/F(t_1)}(Q)\neq0$ (Theorem~\ref{thm:deltaLFnontriv});
\item
  a biquaternion $K(t_1,t_2)$-algebra $B$ such that
  $\delta_{K(t_1,t_2)/F(t_1,t_2)}(B)\neq0$
  (Theorem~\ref{thm:deltaKFnontriv});
\item
  a central simple $F(t_1,t_2,t_3)$-algebra $A$ of degree~$8$ and
  exponent~$2$ such that $\Delta(A)\neq0$
  (Theorem~\ref{thm:Deltanontriv}). 
\end{itemize}
In addition, a biquaternion $K(t_1)$-algebra $B'$ containing $Q$ such
that $\delta_{K(t_1)/F(t_1)}(B')=0$ is given in
Remark~\ref{rem:deltaKFtriv}, and a central simple
$F(t_1,t_2)$-algebra $A'$ of degree~$8$ and exponent~$2$ containing
$B$ such that $\Delta(A')=0$ is given in Remark~\ref{rem:Deltatriv}.

Since $\Delta(A)\neq0$, the algebra $A$ in
Theorem~\ref{thm:Deltanontriv} does not decompose into a tensor
product of quaternion algebras, but it is Brauer-equivalent to a
tensor product of four quaternion algebras, which is particularly
easy to describe explicitly.

Besides Rowen's early examples in~\cite{Rowen}, few methods to
construct indecomposable division algebras of degree~$8$ and
exponent~$2$ in characteristic~$2$ are known. A plausibly exhaustive
list is given by the contributions of Karpenko~\cite{Karp},
McKinnie~\cite{McK} and Aravire--Jacob--O'Ryan~\cite{AJOR}. McKinnie's
examples have the noteworthy property of not containing any inseparable
quadratic extension of the center. In~\cite{AJOR},  Theorem~4.1(iii)
asserts that any central division 
algebra of exponent~$2$ with a
maximal subfield of 
the form $F\bigl(\wp^{-1}(b_1),\wp^{-1}(b_2),\sqrt x\bigr)$ over a field $F$ of
characteristic~$2$  decomposes as a tensor product 
of three quaternion algebras. The algebra from
Theorem~\ref{thm:Deltanontriv} is a counterexample to this assertion. In addition, as explained in the appendix, ~\cite[Theorem~4.1(ii)]{AJOR} fails for purely inseparable quartic field extensions of exponent $2$: remark~\ref{rem:cex} provides an example of an indecomposable algebra of degree $8$ and exponent $2$ which contains a maximal subfield $F\bigl(\wp^{-1}(a),\sqrt[4]{x}\bigr)$.

\subsection*{Notation}

Throughout the paper, $F$ is a field of characteristic~$2$. For
quadratic and bilinear forms over $F$, we use the same notation
as~\cite{EKM}: the Witt group of quadratic forms and the Witt ring of
bilinear forms over $F$ are denoted by $I_q(F)$ and $W(F)$
respectively, and for every integer $n\geq1$ the $n$-th power of the
fundamental ideal $I(F)$ of $W(F)$ is denoted by $I^n(F)$. We write
$I_q^{n+1}(F)$ for $I^n(F)\cdot I_q(F)$ and let
\[
  \overline I_q^n(F) = I^n_q(F)/I^{n+1}_q(F).
\]
For simplicity, we do not distinguish notationally between a form from its Witt-equivalence class. For $\varphi\in
I^n_q(F)$, we write $\overline\varphi$ for the image of $\varphi$ in
$\overline I^n_q(F)$.

Recall from~\cite[Th.~14.3]{EKM} that the Clifford invariant defines a
map $e_2$ from $I^2_q(F)$ to $\Br_2(F)$, the $2$-torsion subgroup of
the Brauer group. A theorem of Sah~\cite{Sah} (see
also~\cite[Th.~16.3]{EKM}) asserts that $e_2$ induces an isomorphism
\[
  \overline e_2\colon \overline I_q^2(F) \xrightarrow{\sim} \Br_2(F).
\]
For any central simple $F$-algebra $A$ such that $2[A]=0$ in $\Br(F)$,
we write $q_A$ for a quadratic form in $I^2_q(F)$ such that
$e_2(q_A)=[A]$. Its image $\overline{q_A}\in\overline I^2_q(F)$ is
thus uniquely determined by $A$. In particular, when $A$ is a
biquaternion algebra, the form $q_A$ can be chosen to be
$6$-dimensional. The $6$-dimensional quadratic forms $q_A$ are known
as \emph{Albert forms} of the biquaternion algebra $A$,
see~\cite[\S16.A]{BoI}.

For any central simple $F$-algebra $A$ and any field extension $L$ of
$F$, we write $A_L$ for the $L$-algebra $A\otimes_FL$ obtained by
scalar extension from $F$ to $L$. We also write $[A]_L$ for
$[A_L]\in\Br(L)$. Similarly, for any quadratic space $(V,\varphi)$
over $F$, we write $(V_L,\varphi_L)$ for the quadratic space over $L$
obtained by extending scalars.

For central simple algebras with involution, we use the notation
of~\cite{BoI}. In particular, for every involution $\sigma$ on a
central simple $F$-algebra $A$, we let
\[
  \Sym(\sigma)=\{x\in A\mid \sigma(x)=x\}
  \qquad\text{and}\qquad
  \Symd(\sigma) = \{x+\sigma(x)\mid x\in A\}.
\]
Recall that an $F$-linear involution $\sigma$ is symplectic if and
only if $1\in\Symd(\sigma)$, see~\cite[(2.6)]{BoI}. The reduced
characteristic polynomial of 
every $x\in\Symd(\sigma)$ is then the square of the \emph{Pfaffian
  characteristic polynomial} $\Prp_{\sigma,x}$,
see~\cite[(2.10)]{BoI}. If $\deg(A)=2m$, then the forms $\Trp_\sigma$,
$\Srp_\sigma$ and $\Nrp_\sigma$ of degree~$1$, $2$, and $m$
respectively are defined on $\Symd(\sigma)$ as coefficients of the
Pfaffian characteristic polynomial by the equation
\[
  \Prp_{\sigma,x}(X) = X^m + \Trp_\sigma(x) X^{m-1} + \Srp_\sigma(x)
  X^{m-2} + \cdots + \Nrp_\sigma(x) \qquad\text{for
    $x\in\Symd(\sigma)$}.
\]
Thus, for all $x\in \Symd(\sigma)$, $\Nrd_A(x)$ is a square and $\Nrp_\sigma(x)$ is its unique square root; it does not depend on the choice of the
symplectic involution $\sigma$ such that $x\in\Symd(\sigma)$.

\section{The invariant $\Delta$}
\label{sec:Delta}

Throughout this section, $A$ is a central simple $F$-algebra of
degree~$8$ whose Brauer class $[A]$ is $2$-torsion. Recall
from~\cite{BGBT3} that a $3$-fold quadratic Pfister form
$\pi(\sigma)$ called the 
\emph{discriminant Pfister form} is attached to each symplectic
involution $\sigma$ on $A$.
We briefly recall the construction in characteristic~$2$, based on the
existence of \'etale 
biquadratic $F$-algebras $L\subset \Symd(\sigma)$ (i.e., $F$-algebras
that carry a Galois structure with elementary abelian Galois group of
order~$4$): let $K\subset L$ be the quadratic \'etale subalgebra fixed
under a nontrivial element $\alpha$ in the Galois group and
\begin{equation*}
  W=\{w\in\Symd(\sigma)\mid w\ell = \alpha(\ell)w \text{ for all
      $\ell\in L$}\}.
\end{equation*}
It is shown in~\cite[Lemma~7.1]{BGBT3} that $\dim_FW=8$ and $w^2\in
K$ for all $w\in W$. Moreover, since the trace map $T_{K/F}$ vanishes on $F$,
\cite[Th.~7.3]{BGBT3} shows that the map $w\mapsto T_{K/F}(w^2)$ for
$w\in W$ is a scalar multiple of a $3$-fold Pfister form. This form
turns out to be independent up to isometry of the choices of $L$ and
$\alpha$, and to be hyperbolic if and only if $A$ decomposes into a
tensor product of three quaternion algebras stable under $\sigma$; it is
the discriminant Pfister form of $\sigma$ according to the definition
in~\cite[p.~269]{BGBT3} or \cite[Prop.~4, Cor.~6]{CRAS}.
Since the characteristic
of $F$ is~$2$, it is shown in~\cite{CRAS} that 
there is a $5$-fold quadratic 
Pfister form $\widehat\pi(\sigma)$ such that
the Witt-equivalence class of
$\Srp_\sigma$ decomposes as
\[
  \Srp_\sigma = [1,1] + \pi(\sigma) + \widehat\pi(\sigma)
  \qquad\text{in $I_q(F)$}.
\]
The discriminant Pfister form $\pi(\sigma)$ is uniquely determined by
this decomposition. 

Our goal in this section is to show
that the discriminant Pfister forms of the symplectic involutions on
$A$ all have the same image in the quotient of $\overline I^3_q(F)$ by
the subgroup $\overline{\bpf{F^\times}q_A}$, where $q_A\in I^2_q(F)$
is any quadratic form such that $e_2(q_A)=[A]$. This image is
therefore an invariant of $A$, which we define to be the
$\Delta$-invariant, see Definition~\ref{def:Delta}.
\medbreak

As in characteristic different from~$2$ (see~\cite[\S2]{GPT}), the
proof proceeds by index reduction to the case where $A$ is not a
division algebra, i.e. $A$ is Brauer-equivalent to a biquaternion
algebra. 
In this case, we have the following: 
\begin{lemma}
  \label{lem:comphyp}
  Suppose $A$ is not a division algebra and let $q_A\in I^2_q(F)$ be
  an Albert form such that $e_2(q_A) = [A]$. Let $\sigma_0$ be a
  hyperbolic symplectic involution on $A$ and let
  $\sigma=\Int(u)\circ\sigma_0$ for some unit
  $u\in\Symd(\sigma_0)$. Then $\pi(\sigma)$ is Witt-equivalent to a
  scalar multiple of $\bpf{\Nrp_{\sigma_0}(u)}q_A$.
\end{lemma}
\begin{proof}
The proof has two steps. First, we prove that the quadratic form $\bpf{\Nrp_{\sigma_0}(u)}q_A$ does not depend on the choices of $\sigma_ 0$ and $u$ as in the statement. We then produce a particular choice of $\sigma_0$ and $u$ for which $\pi(\sigma)$ is Witt-equivalent to a
  scalar multiple of $\bpf{\Nrp_{\sigma_0}(u)}q_A$. 
  
 Assume $\sigma=\Int(u)\circ\sigma_0=\Int(u')\circ\sigma_0'$ for some hyperbolic involutions $\sigma_0$ and $\sigma_0'$ and some units 
  $u\in\Symd(\sigma_0)$ and  $u'\in\Symd(\sigma_0')$. 
   Since all the hyperbolic symplectic involutions on $A$ are
  conjugate, there exists $a\in A^\times$ such that
  \[
    \sigma_0'=\Int(a)\circ\sigma_0\circ\Int(a)^{-1} =
    \Int(a\sigma_0(a))\circ\sigma_0.
    \]
    It follows that $\sigma=\Int(u)\circ\sigma_0=\Int(u'a\sigma_0(a))\circ\sigma_0$, so
    \[u=\lambda u' a\sigma_0(a)\qquad\text{for some $\lambda\in
      F^\times$.}
      \]
      Computing the reduced norms, we get 
    $\Nrd_A(u)=\lambda^8\Nrd_A(u')\Nrd_A(a)^2$, hence
  \[
    \Nrp_{\sigma_0}(u)= \lambda^4\Nrd_A(a)\Nrp_{\sigma_0'}(u')
  \]
  because $\Nrp_{\sigma_0}(u)$ (resp.\ $\Nrp_{\sigma_0'}(u')$) is the unique
  square root of $\Nrd_A(u)$ (resp.\ $\Nrd_A(u')$). Now, $\Nrd_A(a)$ is
  a reduced norm from any biquaternion algebra Brauer-equivalent to
  $A$, hence it is a similarity factor of $q_A$
  by~\cite[(16.6)]{BoI}. Therefore, $\bil{\lambda^4\Nrd_A(a)}q_A\simeq
  q_A$, and we get $\bil{\Nrp_{\sigma_0}(u)}q_A \simeq \bil{\Nrp_{\sigma_0'}(u')}q_A$ as expected.

  To produce a suitable choice of $\sigma_0$ and $u$, choose a representation 
  $A=\End_D(V)$ for some division algebra $D$
  of degree~$1$, $2$ or $4$ and some right $D$-vector space $V$ of
  dimension~$8$, $4$ or $2$ respectively. Then $\sigma$ is adjoint to
  a form $h$ on $V$ that is alternating if $D=F$ and hermitian for a
  symplectic involution on $D$ if $\deg(D)=2$ or $4$. In each case, we
  may consider an orthogonal decomposition
  \[
    (V,h) = (V_1,h_1)\perp (V_2,h_2) \qquad\text{with $\dim V_1=\dim
      V_2$}, 
  \]
  for some nonsingular forms $h_1$ and $h_2$. 
Using a $D$-linear isomorphism $V_1\simeq V_2$, we identify $\End_D(V_1)=\End_D(V_2)$, denoted by $B$, and $A=M_2(B)$.
  Let $v\in B^\times$ be such that $h_2(x,y) =
  h_1(v(x),y)$ for all $x$, $y\in V_2$. Then the involutions
  $\sigma_1$, $\sigma_2$ on $B$ adjoint to $h_1$ and $h_2$ are related
  by
  \[
    \sigma_2=\Int(v^{-1})\circ\sigma_1.
  \]
  These involutions are both symplectic since $h_1$ and $h_2$ have the
  same type as $h$, and the involution $\sigma$ on $M_2(B)$ is given
  by
  \[
    \sigma
    \begin{pmatrix}
      a_{11}&a_{12}\\ a_{21}&a_{22}
    \end{pmatrix}
    =
    \begin{pmatrix}
      \sigma_1(a_{11})&\sigma_1(a_{21})v\\
      v^{-1}\sigma_1(a_{12}) & v^{-1}\sigma_1(a_{22})v
    \end{pmatrix}
    \qquad\text{for $a_{11}$, \dots, $a_{22}\in B$.}
  \]
  Therefore, $\sigma=\Int(u)\circ \sigma_0$, where 
   $u=\bigl(
  \begin{smallmatrix}
    1&0\\ 0&v^{-1}
  \end{smallmatrix}
  \bigr)$ and $\sigma_0$ is the involution given by
  \[
    \sigma_0
    \begin{pmatrix}
      a_{11}&a_{12}\\ a_{21}&a_{22}
    \end{pmatrix}
    =
    \begin{pmatrix}
      \sigma_1(a_{11})&\sigma_1(a_{21})\\
      \sigma_1(a_{12}) & \sigma_1(a_{22})
    \end{pmatrix}
    \qquad\text{for $a_{11}$, \dots, $a_{22}\in B$.}
  \]
  Since $\sigma_1$ is symplectic, there exists $a\in B$ such that
  $\sigma_1(a)+a=1$. Then the element $e=
  \bigl(
  \begin{smallmatrix}
    a&1+a\\ a&1+a
  \end{smallmatrix}
  \bigr)$
  satisfies $\sigma_0(e)+e=1$ and $e^2=e$, which shows that $\sigma_0$ is
  a hyperbolic symplectic involution. Since $\sigma_0$ and $\sigma$ are
  both symplectic involutions, it follows that $u\in\Symd(\sigma_0)$
  by~\cite[(2.7)]{BoI}. To complete the proof, we
  compute $\pi(\sigma)$.

  Observe that $\Symd(\sigma_1)\times \Symd(\sigma_2)$ embeds
  diagonally in $\Symd(\sigma)$. Let $K=F\times F$ diagonally embedded
  in $\Symd(\sigma)$. By~\cite[Th.~6.10]{BGBT1}, there exists a
  biquadratic \'etale $F$-algebra $L\subset\Symd(\sigma)$ containing
  $K$ as the subalgebra fixed under a nontrivial element $\alpha$ of
  its Galois group. This algebra has the form $L=L_0\times L_0$ for
  some quadratic \'etale $F$-algebra $L_0$ that embeds in
  $\Symd(\sigma_1)$ and $\Symd(\sigma_2)$, and
  $\alpha=\alpha_0\times\alpha_0$ for $\alpha_0$ the nontrivial
  $F$-automorphism of $L_0$. Letting
  \[
    W=\{w\in\Symd(\sigma)\mid w\ell=\alpha(\ell)w\text{ for all
      $\ell\in L$}\}
  \]
  and
  \[
    W_i=\{w\in\Symd(\sigma_i)\mid w\ell=\alpha(\ell)w\text{ for all
      $\ell\in L_0$}\}\quad\text{for $i=1$, $2$},
  \]
  we have for the squaring maps $\sq_W\colon W\to K$ and
  $\sq_{W_i}\colon W_i\to F$
  \[
    (W,\sq_W) = (W_1,\sq_{W_1})\times (W_2,\sq_{W_2}),
  \]
  hence by definition $\pi(\sigma)$ is Witt-equivalent to a scalar
  multiple 
  of $\sq_{W_1}\perp\sq_{W_2}$. On the other hand, \cite[Prop.~5.6
  and 6.5]{BGBT1} yields
  \[
    (\Symd(\sigma_i),\Nrp_{\sigma_i})\simeq (L_0,N_{L_0/F})\perp
    (W_i,\sq_{W_i})
    \qquad\text{for $i=1$, $2$,}
  \]
  where $N_{L_0/F}$ is the norm quadratic
  form. Therefore, $\sq_{W_1}\perp\sq_{W_2}$ is Witt-equivalent to
  $\Nrp_{\sigma_1}\perp\Nrp_{\sigma_2}$. Since
  $\sigma_1=\Int(v)\circ\sigma_2$, it follows from~\cite[(16.18)]{BoI}
  that $\Nrp_{\sigma_1}\cong \bil{\Nrp_{\sigma_2}(v)}\Nrp_{\sigma_2}$,
  hence $\pi(\sigma)$ is Witt-equivalent to a scalar multiple of
  $\bpf{\Nrp_{\sigma_2}(v)}\Nrp_{\sigma_2}$. This completes the proof
  because $\Nrp_\sigma(u)=\Nrp_{\sigma_2}(v)$ and $\Nrp_{\sigma_2}$ is
  an Albert form of $B$ by~\cite[(16.8)]{BoI}, hence $q_A$ is similar
  to $\Nrp_{\sigma_2}$ by~\cite[(16.3)]{BoI}.
\end{proof} 

Without restriction on the index of $A$, let $q_A\in I^2_q(F)$ be such
that $e_2(q_A)=[A]$. Note that $q_A$ is uniquely determined modulo
$I^3_q(F)$, hence $\overline{q_A}\in\overline I^2_q(F)$ depends only
on $A$.

\begin{thm}
  \label{thm:main}
  Let $\sigma$, $\sigma'$ be symplectic involutions on $A$ and let
  $w\in\Symd(\sigma)$ be such that $\sigma'=\Int(w)\circ\sigma$. Then
  \[
    \overline{\pi(\sigma')} = \overline{\bpf{\Nrp_\sigma(w)}q_A} +
    \overline{\pi(\sigma)} \qquad\text{in $\overline I^3_q(F)$}.
  \]
\end{thm}

\begin{proof}
  If $A$ is not a division algebra, we may choose for $q_A$ an Albert
  form and pick a hyperbolic symplectic involution $\sigma_0$ on
  $A$. If $u\in\Symd(\sigma_0)$ satisfies
  $\sigma=\Int(u)\circ\sigma_0$, then $\sigma'=\Int(wu)\circ\sigma_0$
  and Lemma~\ref{lem:comphyp} shows that $\pi(\sigma)$ and
  $\pi(\sigma')$ are Witt-equivalent to scalar multiples of
  $\bpf{\Nrp_{\sigma_0}(u)} q_A$ and $\bpf{\Nrp_{\sigma_0}(wu)}q_A$
  respectively, hence
  \[
    \overline{\pi(\sigma)} = \overline{\bpf{\Nrp_{\sigma_0}(u)}q_A}
    \quad\text{and}\quad
    \overline{\pi(\sigma')} =
    \overline{\bpf{\Nrp_{\sigma_0}(wu)}q_A}
    \quad\text{in $\overline I^3_q(F)$}.
  \]
  By taking the square root of each side in
  $\Nrd_A(wu)=\Nrd_A(w)\Nrd_A(u)$ we obtain $\Nrp_{\sigma_0}(wu)=
  \Nrp_{\sigma}(w)\Nrp_{\sigma_0}(u)$, hence
  $\bpf{\Nrp_{\sigma_0}(wu)} \equiv \bpf{\Nrp_{\sigma}(w)}+
  \bpf{\Nrp_{\sigma_0}(u)}\bmod I^2(F)$. Therefore,
  \[
    \overline{\pi(\sigma')} =
    \overline{\bpf{\Nrp_{\sigma}(w)}q_A}
    +
    \overline{\bpf{\Nrp_{\sigma_0}(u)}q_A}
    =
    \overline{\bpf{\Nrp_{\sigma}(w)}q_A}+
    \overline{\pi(\sigma)}.
  \]
  Thus, the theorem is proved in the case where $A$ is not a division
  algebra.

  When $A$ is a division algebra, we reduce the index by scalar
  extension to use the preceding case. Recall from~\cite[Th.~2]{Rowen}
  that $A$ is Brauer-equivalent to a tensor product of four quaternion
  algebras; let
  \[
    [A]=[Q_1\otimes_FQ_2\otimes_FQ_3\otimes_FQ_4] \quad\text{in
      $\Br(F)$.}
  \]
  Let $F_1$ be the function field of an Albert quadratic form of
  $Q_3\otimes_FQ_4$. (If $Q_3\otimes_FQ_4$ is not a division algebra,
  then the Albert quadratic form is isotropic and $F_1$ is a rational
  function field over $F$.) After scalar extension from $F$ to $F_1$,
  $Q_3\otimes_FQ_4$ is 
  Brauer-equivalent to a quaternion algebra $H$, hence
  \[
    [A_{F_1}] = [Q_1\otimes_FQ_2\otimes_FH] \quad\text{in $\Br(F_1)$.}
  \]
  Extend scalars further from $F_1$ to the function field $F_2$ of an
  Albert quadric of $Q_2\otimes_FH$; then $A_{F_2}$ is not a division
  algebra, hence by the special case of the theorem proved above,
  \[
    \overline{\pi(\sigma')} - \overline{\bpf{\Nrp_\sigma(w)}q_A} -
    \overline{\pi(\sigma)}\in \ker\bigl(\overline I^3_q(F) \to
    \overline I^3_q(F_2)\bigr).
  \]
  But by~\cite[Th.~5.6]{HLS}\footnote{We are indebted to A.~Laghribi
    for pointing out this reference to us.} scalar extension to the
  function field of an Albert quadratic form induces an injective map
  on $\overline I^3_q(F)$, hence $\ker\bigl(\overline
  I^3_q(F) \to \overline I^3_q(F_2)\bigr)=0$ and the proof is
  complete. 
\end{proof}
  
Theorem~\ref{thm:main} allows us to set the following definition:

\begin{defn}
  \label{def:Delta}
  For a central simple $F$-algebra $A$ of degree~$8$ with $2[A]=0$ in
  $\Br(F)$ and any quadratic form $q_A\in I^2_q(F)$ such that
  $e_2(q_A)=[A]$, the invariant $\Delta(A)$ is the image in the
  quotient $\overline I^3_q(F)/\overline{\bpf{F^\times}q_A}$ of the
  discriminant Pfister form of any symplectic involution $\sigma$ on
  $A$,
  \[
    \Delta(A) = \overline{\pi(\sigma)} +
    \overline{\bpf{F^\times}q_A} \in \overline
    I_q^3(F)/\overline{\bpf{F^\times}q_A}.
  \]
\end{defn}

Vanishing of $\Delta(A)$ is a necessary condition for the
decomposability of $A$:

\begin{prop}
  \label{prop:decomp}
  If $A=Q_1\otimes_FQ_2\otimes_FQ_3$ for some quaternion $F$-algebras
  $Q_1$, $Q_2$, $Q_3$, then $\Delta(A)=0$.
\end{prop}

\begin{proof}
  The tensor product of the canonical involutions on $Q_1$, $Q_2$ and
  $Q_3$ is a symplectic involution with hyperbolic discriminant
  Pfister form, by the main theorem in~\cite{BGBT3} or
  \cite[Prop.~8]{CRAS}. 
\end{proof}

\subsection*{Cohomological $\Delta$-invariant}

Translating the construction above into cohomological terms is
straightforward. Let $\Omega_F^1$ be the $F$-vector space of absolute
differentials on $F$ and $\Omega_F^n = \bigwedge^n\Omega_F^1$ for
$n\geq2$. Recall that the (Kato--Milne) cohomology group $H^{n+1}(F)$
is defined as the cokernel of the map
\[
  \wp\colon \Omega^n_F \to \Omega_F^n/d(\Omega_F^{n-1})
  \quad\text{such that } \wp\bigl( x \frac{dy_1}{y_1}\wedge \ldots
  \wedge \frac{dy_n}{y_n} \bigr) = (x^2-x) \frac{dy_1}{y_1}\wedge \ldots
  \wedge \frac{dy_n}{y_n}
\]
for $x\in F$ and $y_1$, \ldots, $y_n\in F^\times$, see~\cite{Kato}. A
different (but equivalent) definition is given in~\cite[\S101]{EKM},
where $H^{n+1}(F)$ is denoted by $H^{n+1,n}(F, \Z/2\Z)$. Kato proves
in~\cite{Kato} that the map carrying $x \frac{dy_1}{y_1}\wedge \ldots
\wedge \frac{dy_n}{y_n}\in\Omega_F^n$ to $\overline{\qpf{y_1,\ldots,
y_n,x}}\in \overline I_q^{n+1}(F)$ for $x\in F$ and $y_1$,
\ldots, $y_n\in F^\times$ induces a canonical isomorphism
\[
  \kappa_{n+1}\colon H^{n+1}(F) \stackrel{\sim}{\longrightarrow}
  \overline I^{n+1}_q(F).
\]
In particular, the composition $\overline e_2\circ\kappa_2\colon
H^2(F)\to \Br_2(F)$ is an isomorphism, hence for every central simple
$F$-algebra $A$ such that $2[A]=0$ we may find $h_A\in H^2(F)$ such
that $\overline e_2\circ\kappa_2(h_A)=[A]$. For $x\in F^\times$, let
$\dlog(x) = \frac{dx}{x}\in\Omega_F^1$. The exterior product defines a
map $F^\times\times H^2(F)\to H^3(F)$ carrying $(x,h_A)$ to
$\dlog(x)\wedge h_A$. By definition, $\kappa_3(\dlog(x)\wedge h_A) =
\overline{\bpf{x}\kappa_2(h_A)}$, and $\kappa_2(h_A) = \overline q_A$
for any quadratic form $q_A\in I^2_q(F)$ such that $e_2(q_A)=[A]$,
hence $\kappa_3$ defines an isomorphism
\begin{equation}
  \label{eq:iso}
  H^3(F)/\dlog(F^\times)\wedge h_A \stackrel{\sim}{\longrightarrow}
  \overline I_q^3(F)/\overline{\bpf{F^\times}q_A}.
\end{equation}
If $\deg(A)=8$ and $\sigma$ is a symplectic involution on $A$, let
$\pi_h(A,\sigma)\in H^3(F)$ be defined by the equation
\[
  \kappa_3\bigl(\pi_h(A,\sigma)\bigr) = \overline{\pi(\sigma)} \in
  \overline I_q^3(F).
\]
Theorem~\ref{thm:main} readily shows that if $\sigma$, $\sigma'$ are
symplectic involutions on $A$ and $w\in\Symd(\sigma)$ is such that
$\sigma'=\Int(w)\circ\sigma$, then
\begin{equation}
  \label{eq:rel}
  \pi_h(A,\sigma') = \dlog(w)\wedge h_A + \pi_h(A,\sigma).
\end{equation}
Therefore, we can mimic Definition~\ref{def:Delta} by setting
\[
  \Delta_h(A) = \pi_h(A,\sigma) + \dlog(F^\times)\wedge h_A \in
  H^3(F)/\dlog(F^\times)\wedge h_A.
\]
From the definitions, it is clear that the
isomorphism~\eqref{eq:iso} maps $\Delta_h(A)\in
H^3(F)/\dlog(F^\times)\wedge h_A$ to $\Delta(A)\in \overline
I^3_q(F)/\overline{\bpf{F^\times}q_A}$.

The discussion in~\cite[\S4a]{MNZ} can now be reproduced without any
restriction on the characteristic: 
\eqref{eq:rel} shows that $\pi_h$ defines a nondecomposable
cohomological invariant of the adjoint symplectic group
$\mathbf{PGSp}_8$, and $\Delta_h$ is nontrivial generically.

\section{Descent of Brauer classes}
\label{sec:quadesc}

Throughout this section, $K$ denotes a quadratic extension of an
arbitrary field $F$ (of characteristic~$2$). We write $N_{K/F}$ for
the norm map $N_{K/F}\colon K^\times\to 
F^\times$ (which is the squaring map $\sq_K$ when $K$ is inseparable
over $F$), and also for the map $N_{K/F}\colon \Br_2(K)\to \Br_2(F)$
that is given by corestriction when $K$ is separable over $F$, and that is
functorially 
induced by the squaring field homomorphism $\sq_K\colon K\to F$ when
$K$ is inseparable over $F$ (see~\cite[\S3.13, \S4.1]{J} or
\cite[\S3]{Bingol}).

\begin{thm}
  \label{thm:quadesc}
  Let $E$ be a central simple $K$-algebra such that $2[E]=0$ in
  $\Br(K)$. There exists a central simple $F$-algebra $E'$ such that
  $2[E']=0$ and $[E']_K=[E]$ if and only if $N_{K/F}[E]=0$. When this
  condition holds, the algebra $E'$ can be chosen of degree
  $\deg(E')=2\deg(E)$. 
\end{thm}

The analogous statement in characteristic different from~$2$ is
well-known, see~\cite[Th.~99.13]{EKM}.

\begin{proof}
  Suppose first that $K$ is separable over $F$, and let $\iota$ denote
  the nontrivial $F$-automorphism of $K$. If $[E']_K=[E]$ for some
  central simple $F$-algebra $E'$, then $N_{K/F}[E]=2[E']$
  by~\cite[(3.13)(5)]{BoI}. The ``only if'' part of the statement
  follows.

  For the converse, recall from~\cite[(3.1)]{BoI} that when $2[E]=0$
  and $N_{K/F}[E]=0$, then $E$ carries an involution $\sigma$ of the
  first kind (i.e., whose restriction to $K$ is the identity) and an
  involution $\tau$ of the second kind, whose restriction to $K$ is
  $\iota$. 
  The composition $\tau\circ\sigma\circ\tau$ is an involution of the first kind of $E$, therefore $\tau\circ\sigma\circ\tau=\Int(u)\circ \sigma$ for some invertible element $u\in E^\times$ satisfying $\sigma(u)=u$. 
  It follows that the $\iota$-semilinear automorphism $\tau\circ \sigma$ of $E$ satisfies $(\tau\circ\sigma)^2=\Int(u)$; 
  in particular, it commutes with $\Int(u)$. 
  Hence, $\Int\bigl(\tau\circ\sigma(u)\bigr)=\Int(u)$, so that 
  $\tau(u)=\tau\circ\sigma(u)=\lambda u$ for some $\lambda \in K^\times$. 
Applying $\tau$ to the equality $\tau(u)=\lambda u$, we get that $\lambda\iota(\lambda)=1$.
Hilbert's Theorem~90 yields $\mu\in K^\times$
  such that $\lambda = \mu \iota(\mu)^{-1}$; substituting $\mu u$ for
  $u$, we may assume $\sigma(u)=\tau(u)=u$. Define an $F$-algebra
  $E'$ as a ``generalized crossed product''
  \[
    E'=E\oplus Ez
  \]
  where the multiplication extends the multiplication in $E$ by the
  following rules:
  \[
    z^2=u\qquad\text{and}\qquad
    zx=\sigma\circ\tau(x)z \qquad\text{for $x\in E$.}
  \]
  By~\cite[Th.~11.11]{Albert}, $E'$ is a central simple $F$-algebra
  containing $E$ as the centralizer of $K$, hence
  $[E']_K=[E]$. Moreover, setting $\sigma'(z)=z$ defines an involution
  $\sigma'$ on $E'$ extending $\sigma$ on $E$, hence $2[E']=0$. The
  theorem is thus proved in the case where $K$ is separable over $F$.

  Now, suppose $K$ is inseparable over $F$. By a theorem of Hochschild
  (see~\cite[Th.~4.1.5]{J}), $[E]=[E']_K$ for some central simple
  $F$-algebra $E'$. Since $N_{K/F}$ is induced by the squaring map, we
  have $N_{K/F}[E] = 2[E']$ (see~\cite[Th.~4.1.2]{J}). This proves the first part of the statement.

  Note that every field extension of $K$ that splits $E$ also splits
  $E'$, hence $E'$ is split by field extensions of $F$ of degree
  dividing 
  $2\deg(E)$. Therefore, the central simple $F$-algebra $E'$ may be
  assumed to have $\deg(E')=2\deg(E)$. (This is also clear 
  from the proof above in the case where $K$ is separable over $F$ or
  from the proof of Hochschild's theorem in \cite{J}.)
\end{proof}

When $E$ is a quaternion algebra, Theorem~\ref{thm:quadesc} can be
improved as follows: 

\begin{prop}
  \label{prop:descquat}
  Let $Q$ be a quaternion algebra over $K$.
  \begin{enumerate}
  \item[(a)]
    If $N_{K/F}[Q]=0$, then there exists a quaternion $F$-algebra
    $Q_0$ such that $Q\simeq Q_0\otimes_FK$.
  \item[(b)]
    If $Q=[\alpha,\beta)_K$ with $\alpha\in F$ and $
    N_{K/F}[Q]=0$, then there exists $\gamma\in
    F^\times$ such that $Q\simeq[\alpha,\gamma)_F\otimes_FK$.
  \end{enumerate}
\end{prop}

\begin{proof}
  We first prove~(b). The argument depends on whether $K$ is separable
  over $F$ or not. By assumption, the algebra $[\alpha,N_{K/F}(\beta))$ is split.
  
  Assume first $K$ is separable, and let $\iota$
  denote the 
  nontrivial $F$-automorphism of $K$. If $\alpha\in\wp(K)$, then
  $[\alpha,\beta)_K$ is split and we may choose $\gamma=1$. If
  $\alpha\notin\wp(K)$, then $K(\wp^{-1}(\alpha))$ is a biquadratic
  separable 
  field extension of $F$ and $N_{K/F}(\beta)$ is a norm for the quadratic separable extension $F(\wp^{-1}(\alpha))/F$. By the ``biquadratic lemma''
  \cite[Exemple~5.1]{Tignol}, there exist
  $\lambda\in F^\times$ and $\mu\in K(\wp^{-1}(\alpha))$ such that
  \[
    N_{K/F}(\beta) = \lambda^2N_{K(\wp^{-1}(\alpha))/F}(\mu) =
    N_{K/F}\bigl(\lambda N_{K(\wp^{-1}(\alpha))/K}(\mu)\bigr).
  \]
  Hilbert's Theorem~90 then yields $\nu\in K^\times$ such that
  \[
    \beta=\lambda N_{K(\wp^{-1}(\alpha))/K}(\mu) \nu\,\iota(\nu)^{-1}
    = \lambda N_{K/F}(\nu)^{-1} \nu^2 N_{K(\wp^{-1}(\alpha))/K}(\mu)
  \]
  and (b) holds with $\gamma=\lambda N_{K/F}(\nu)^{-1}$.

  Now, assume $K$ is inseparable over $F$. Since
  $[\alpha,N_{K/F}(\beta))_F$ is split, there exist $\lambda$, 
  $\mu\in F$ such that
  \[
    N_{K/F}(\beta)=\beta^2=\lambda^2+\lambda\mu+\alpha\mu^2.
  \]
  If $\mu=0$, then $\beta=\lambda$
  and~(b) holds with $\gamma=\lambda$. If $\mu\neq0$, then
  \[
   \beta^2=\mu^{-2}\bigl((\lambda+\beta)^2+(\lambda+\beta)\mu +
    \alpha\mu^2\bigr)^2,
    \quad\text{hence}\quad
    \beta=\mu^{-1}\bigl((\lambda+\beta)^2+(\lambda+\beta)\mu +
    \alpha\mu^2\bigr),
  \]
  and (b) holds with $\gamma=\mu^{-1}$.

  We deduce (a) from~(b) when $K$ is inseparable over $F$, as follows:
  we write $Q=[\alpha_0,\beta)_K$ for some $\alpha_0\in K$ and
  $\beta\in K^\times$, and observe that $Q=[\alpha_0^2,\beta)_K$ with
  $\alpha_0^2\in F$ because $\alpha_0^2\equiv\alpha_0\bmod\wp(K)$ and
  $K^2\subset F$. Since $N_{K/F}[Q]=[[\alpha_0^2,N_{K/F}(\beta))_F]=0$,
  (b) yields $\gamma\in F^\times$ such that
  $Q\simeq[\alpha_0^2,\gamma)_F\otimes_FK$.

  When $K$ is separable over $F$, the hypothesis $N_{K/F}[Q]=0$
  implies that $Q$ carries an involution of unitary type
  by~\cite[(3.1)]{BoI}. The existence of a quaternion $F$-algebra
  $Q_0$ such that $Q\simeq Q_0\otimes_FK$ then follows from a theorem
  of Albert \cite[(2.22)]{BoI}.
\end{proof}

The last result in this section characterizes the Brauer classes of
quaternion algebras that satisfy a biquadratic descent property. By
definition, the biquadratic field extensions we consider can be
represented as tensor products of two quadratic extensions. We exclude
the case of purely inseparable field extensions, whose features are
significantly different: see the Appendix. Thus, the
biquadratic field extensions $L/F$ 
we consider are either separable---and then Galois with elementary
abelian Galois group---or of the form $L\simeq K_1\otimes_FK_2$ where
$K_1$ is a separable quadratic extension of $F$ and $K_2$ is an
inseparable quadratic extension of $F$. In the latter case, $K_1$ is
uniquely determined as the separable closure of $F$ in $L$, and it is
easily verified that $K_2=\{\ell\in L\mid \ell^2\in F\}$, hence $K_1$
and $K_2$ are the only two quadratic extensions of $F$ in $L$. By
contrast, when $L$ is separable over $F$ it contains three quadratic
extensions of $F$.

\begin{thm}
  \label{thm:biquaddescBr}
  Let $L$ be a biquadratic field extension of $F$ and let $Q$ be a
  quaternion algebra over $L$. If $L$ is not purely inseparable over
  $F$, the following are equivalent:
  \begin{enumerate}
  \item[(i)]
    there exists a central simple $F$-algebra $A$ such that $2[A]=0$
    and $[A]_L=[Q]$;
  \item[(ii)]
    $N_{L/K}[Q]=0$ for every quadratic extension $K$ of $F$ in $L$.
  \end{enumerate}
  When these conditions hold, the algebra $A$ can be chosen of
  degree~$8$, so that $Q$ is isomorphic to the centralizer of $L$ in $A$. 
\end{thm}

The proof below uses the following ``inseparable common slot lemma''
(a ``separable common slot lemma'' is given
in~\cite[Lemma~3]{Rowen}): 

\begin{lemma}
  \label{lem:inscomslot}
  Let $\alpha_1$, $\alpha_2\in F$ and $\beta_1$, $\beta_2\in F^\times$
  be such that the quaternion $F$-algebras $[\alpha_1,\beta_1)_F$ and
  $[\alpha_2,\beta_2)_F$ are isomorphic. There exists $\gamma\in
  F^\times$ such that
  \[
    [\alpha_1,\beta_1)_F\simeq [\alpha_1,\gamma)_F \simeq
    [\alpha_2,\gamma)_F \simeq [\alpha_2,\beta_2)_F.
  \]
\end{lemma}

\begin{proof}
  Let $Q$ be a quaternion $F$-algebra isomorphic to
  $[\alpha_1,\beta_1)_F$ and $[\alpha_2,\beta_2)_F$.
  We may assume $Q$ is a division algebra, for otherwise we may
  readily choose $\gamma=1$. By hypothesis, there exist $i_1$, $j_1$,
  $i_2$, $j_2\in Q$ such that
  \[
    i_r^2+i_r=\alpha_r,\qquad
    j_r i_r+i_r j_r = j_r
    \quad\text{and}\quad j_r^2=\beta_r\qquad\text{for
      $r=1$, $2$.}
  \]
  The vector space of reduced trace zero elements in $Q$ has
  dimension~$3$ and contains the $2$-dimensional subspaces $F(i_1)j_1$
  and $F(i_2)j_2$. By dimension count, these subspaces have a nonzero
  intersection. If $j\in F(i_1)j_1\cap F(i_2)j_2$ is nonzero, then the
  lemma holds with $\gamma=j^2$.
\end{proof}

\begin{proof}[Proof of Theorem~\ref{thm:biquaddescBr}]
  (i)~$\Rightarrow$~(ii)
  Assuming~(i), we have $[A_{K}]_L=[Q]$ for every quadratic extension
  $K$ of $F$ in $L$, hence
  Theorem~\ref{thm:quadesc} shows that~(ii) holds.

  (ii)~$\Rightarrow$~(i)
  The proof depends on whether $L$ is separable over $F$ or
  not. Assume first $L$ is separable, and let
  $L=F\bigl(\wp^{-1}(\alpha_1), \wp^{-1}(\alpha_2)\bigr)$ for some
  $\alpha_1$, $\alpha_2\in F$. The
  quadratic extensions of $F$ in $L$ are then
  $K_1=F\bigl(\wp^{-1}(\alpha_1)\bigr)$,
  $K_2=F\bigl(\wp^{-1}(\alpha_2)\bigr)$ and
  $K_3=F\bigl(\wp^{-1}(\alpha_1+\alpha_2)\bigr)$. Assuming
  $N_{L/K_3}[Q]=0$, Proposition~\ref{prop:descquat}(a) yields a
  quaternion $K_3$-algebra $Q_3$ such that $Q=Q_3\otimes_{K_3}L$. Then
  $N_{L/K_i}[Q] = \bigl(N_{K_3/F}[Q_3]\bigr)_{K_i}$ for $i=1$, $2$,
  hence~(ii) shows that $N_{K_3/F}[Q_3]$ is split by $K_1$ and
  $K_2$. By the inseparable common slot lemma~\ref{lem:inscomslot},
  there exists $\gamma\in F^\times$ such that
  \begin{equation}
    \label{eq:biquadsepdesc}
    N_{K_3/F}[Q_3] = [[\alpha_1,\gamma)_F] = [[\alpha_2,\gamma)_F].
  \end{equation}
  Then $[\alpha_1+\alpha_2,\gamma)_F$ is split, hence there exists
  $\beta\in K_3^\times$ such that
  $\gamma=N_{K_3/F}(\beta)$. From~\eqref{eq:biquadsepdesc} it follows
  that $N_{K_3/F}[Q_3] = N_{K_3/F}[[\alpha_1,\beta)_{K_3}]$, hence
  \[
    N_{K_3/F}[Q_3\otimes_{K_3}[\alpha_1,\beta)_{K_3}] = 0.
  \]
  Theorem~\ref{thm:quadesc} yields a central simple $F$-algebra $A$ of
  degree~$8$ 
  such that $2[A]=0$ and $[A]_{K_3} =
  [Q_3\otimes_{K_3}[\alpha_1,\beta)_{K_3}]$. We
  may then identify $Q_3\otimes_{K_3}[\alpha_1,\beta)_{K_3}$ with the
  centralizer of $K_3$ in $A$. Identifying $L$ with
  $K_3\bigl(\wp^{-1}(\alpha_1)\bigr)\subset [\alpha_1,\beta)_{K_3}$,
  we see that the centralizer of $L$ in $A$ is
  $Q_3\otimes_{K_3}L$. This concludes the proof in the case where $L$
  is separable over $F$ since
  $Q=Q_3\otimes_{K_3}L$.
  \medbreak

  Now, assume $L=F\bigl(\wp^{-1}(\alpha), \sqrt\rho\bigr)$ for some
  $\alpha\in F$ and $\rho\in F^\times$, hence the quadratic extensions
  of $F$ in $L$ are
  $K_1=F\bigl(\wp^{-1}(\alpha)\bigr)$ and $K_2=F(\sqrt\rho)$. Since
  $N_{L/K_2}[Q]=0$, 
  Proposition~\ref{prop:descquat} yields a quaternion $K_2$-algebra
  $Q_2$ such that $Q=Q_2\otimes_{K_2}L$. Then
  $N_{L/K_1}[Q]=\bigl(N_{K_2/F}[Q_2]\bigr)_{K_1}$, hence (ii) shows
  that $N_{K_2/F}[Q_2]$ is split by $K_1$. But
  $\bigl(N_{K_2/F}[Q_2]\bigr)_{K_2} = 2[Q_2]=0$, hence
  $N_{K_2/F}[Q_2]$ is also split by $K_2$. Therefore, $N_{K_2/F}[Q_2]$
  is the Brauer class of a quaternion $F$-algebra $Q_0$ containing
  $K_1$ and $K_2$ as maximal subfields. Let $i$, $r\in Q_0$ be such
  that $i^2+i=\alpha$ and $r^2=\rho$, and let $j=ir+ri\in Q_0$. If
  $j=0$, then $i$ commutes with $r$ hence $i\in F(r)$. This is
  impossible since $i$ is separable over $F$ whereas $F(r)=K_2$ is an
  inseparable extension; hence $j\neq0$. Since $r^2\in F^\times$, it
  is readily seen that $j$ commutes with $r$, hence $j\in
  F(r)=K_2$. On the other hand, computation yields $ij+ji=j$, hence
  $j\notin F$ and $i$, $j$ generate $Q_0$ as a quaternion $F$-algebra
  $[\alpha,\beta)_F$ with $\beta={j}^2\in F^\times$ and
  $K_2=F(j)$. Now, $N_{K_2/F}[[\alpha,j)_{K_2}] = [[\alpha,\beta)_F]
  = N_{K_2/F}[Q_2]$, hence
  \[
    N_{K_2/F}[Q_2\otimes_{K_2}[\alpha,j)_{K_2}] = 0.
  \]
  Theorem~\ref{thm:quadesc} yields a central simple $F$-algebra $A$ of
  degree~$8$ 
  such that $2[A]=0$ and
  $[A]_{K_2}=[Q_2\otimes_{K_2}[\alpha,j)_{K_2}]$. We may assume
  $A$ contains $K_2$, with centralizer
  $Q_2\otimes_{K_2}[\alpha,j)_{K_2}$. Identifying $L$ with
  $K_2\bigl(\wp^{-1}(\alpha)\bigr)\subset [\alpha,j)_{K_2}$, we see
  that the centralizer of $L$ in $A$ is $Q_2\otimes_{K_2}L$. This
  completes the proof since $Q_2\otimes_{K_2}L=Q$.
\end{proof}

\section{Quadratic descent of biquaternion algebras}
\label{sec:Quadbiquat}

Throughout this section, $K$ denotes a quadratic extension of the 
field $F$ of characteristic~$2$, and $B$ is a biquaternion
$K$-algebra. Theorem~\ref{thm:quadesc} yields a necessary and
sufficient condition for the Brauer class $[B]$ to descend to
$\Br_2(F)$, but even when this condition holds the algebra $B$ itself
may not descend to a biquaternion $F$-algebra. 
Our goal in this section is to define an invariant $\delta_{K/F}(B)$ that vanishes if and only if this stronger property holds. As the definition of the $\Delta$-invariant in Section~\ref{sec:Delta}, our definition of $\delta_{K/F}(B)$ relies on discriminant Pfister forms of symplectic involutions on algebras of degree~$8$. It turns out to be equivalent to the definition based on Albert forms introduced by Barry--Chapman--Laghribi \cite[\S3]{BaCL} when $K$ is separable over $F$ and by Laghribi--Lorenz \cite[Prop.~6.4]{LL} when $K$ is inseparable.

\medbreak

We use the same notation $N_{K/F}$ as in the preceding section, and
fix a nonzero $F$-linear functional $t_K\colon K\to F$ such that
$t_K(1)=0$. (If $K$ is separable over $F$, we may take for $t_K$ the
trace map $T_{K/F}$.)  
Throughout, we assume $N_{K/F}[B]=0$, so Theorem~\ref{thm:quadesc}
shows that there exists a central simple $F$-algebra $A$ of degree~$8$
such that 
$[B]=[A]_K$ and $2[A]=0$. Let $\sigma_B$
be a symplectic involution on $B$. By~\cite[(4.14)]{BoI}, $\sigma_B$
extends (not uniquely) to an involution $\sigma_A$ on $A$, which is necessarily
symplectic since $1\in\Symd(\sigma_B) \subset\Symd(\sigma_A)$.
The following result shows that the discriminant Pfister form
$\pi(\sigma_A)$ can be determined from 
$\sigma_B$. In particular, it does not depend on the choices of $A$ and $\sigma_A$.

\begin{prop}
  \label{prop:sNrp}
  The transfer $(t_K)_*(\Nrp_{\sigma_B})$ is Witt-equivalent to a scalar
  multiple of $\pi(\sigma_A)$. 
\end{prop}

\begin{proof}
  The argument branches into two parts, depending on whether $K$ is
  separable over $F$ or not. First, assume $K$ is
  separable. By~\cite[Th.~6.10]{BGBT1}, there exists a
  biquadratic \'etale
  $F$-algebra $L\subset\Symd(\sigma_A)$ containing $K$ as a subalgebra
  fixed under an element of order~$2$ in its Galois group. Let
  $\alpha$ denote this order~$2$ element and define $W$ as previously:
  \[
    W=\{w\in\Symd(\sigma_A)\mid w\ell = \alpha(\ell)w\text{ for all
      $\ell\in L$}\}.
  \]
  Note that $W\subset\Symd(\sigma_B)$ because $K$ is fixed under
  $\alpha$ and $\Symd(\sigma_B)=B\cap\Symd(\sigma_A)$
  by~\cite[Prop.~6.1]{BGBT1}. Letting $\sq_W\colon W\to K$ denote the
  squaring map 
  $\sq_W(w)=w^2$, the transfer $(t_K)_*(\sq_W)$ is a scalar multiple of
  $\pi(\sigma_A)$, by definition of the discriminant Pfister
  form. By~\cite[Prop.~5.6 and 6.5]{BGBT1}, we have
  \[
    (\Symd(\sigma_B),\Nrp_{\sigma_B}) = (L, N_{L/K}) \perp (W,\sq_W),
  \]
  hence $(t_K)_*(\Nrp_{\sigma_B}) = (t_K)_*(N_{L/K}) \perp
  (t_K)_*(\sq_W)$. But 
  $L$ is a biquadratic $F$-algebra, hence the norm form $N_{L/K}$ is
  defined over $F$, and $(t_K)_*(N_{L/K})$ is hyperbolic. Thus, the proof
  is complete in the case where $K$ is separable over $F$.
  \renewcommand{\qed}{\relax}
\end{proof}

In the case where $K$ is inseparable over $F$, we need two lemmas.

\begin{lemma}
  \label{lem:insep1}
  Assume $K$ is an inseparable quadratic extension of $F$. There
  exists $i_1\in \Symd(\sigma_B)$ such that $i_1^2+i_1\in F$. For any
  such $i_1$, the $K$-vector space
  \[
    V=\{v\in\Symd(\sigma_B)\mid vi_1=(i_1+1)v\}
  \]
  has dimension~$4$ and the squaring map $\sq_V\colon v\mapsto v^2$ is
  a nonsingular quadratic form $V\to K$ whose transfer
  $(t_K)_*(\sq_V)$ is Witt-equivalent to $(t_K)_*({\Nrp}_{\sigma_B})$.
\end{lemma}

\begin{proof}
  Consider the $5$-dimensional $K$-vector space $\Symd(\sigma_{B})^0
  := 
  \ker\Trp_{\sigma_B}$. By~\cite[(15.16)]{BoI}, the restriction
  $\varphi$ of $\Nrp_{\sigma_B}$ to $\Symd(\sigma_B)^0$ is a
  nondegenerate quadratic form, and the inclusion
  $\Symd(\sigma_B)^0\hookrightarrow B$ induces a canonical
  isomorphism  
  of algebras with involution from the even Clifford algebra
  $C_0(\varphi)$ with its canonical involution $\gamma$ to
  $(B,\sigma_B)$. If $e_1$, $f_1$ and $e_2$, $f_2$ are
  symplectic bases of two mutually orthogonal nonsingular planes in
  $\Symd(\sigma_B)^0$, then in the even Clifford algebra
  $e_1f_1+f_1e_1=e_2f_2+f_2e_2=1$, and $e_1f_1$ and $e_2f_2$ commute. Hence, we have
  \begin{multline}
    \label{eq:clifford}
    \gamma(e_1f_1e_2f_2) = f_2e_2f_1e_1 = f_1e_1+e_2f_2+e_1f_1e_2f_2\\
    \text{and}\qquad
    (f_1e_1+e_2f_2)^2 = f_1e_1+e_2f_2+
    \varphi(e_1)\varphi(f_1)+\varphi(e_2)\varphi(f_2).
  \end{multline}
  Let $i_1\in B\setminus K$ be the image of $f_1e_1+e_2f_2$ under the
  canonical 
  isomorphism $C_0(\varphi)\simeq B$. From~\eqref{eq:clifford} it
  follows 
  that $i_1\in\Symd(\sigma_B)$ and $i_1^2+i_1\in K$. Substituting
  $i_1^2$ for $i_1$, we may (and will) assume that $i_1^2+i_1\in F$,
  hence $F(i_1)\subset\Symd(\sigma_B)$ is a quadratic \'etale
  $F$-algebra. Since $\dim_KK(i_1)=\frac12\deg B$, Proposition~5.6
  in~\cite{BGBT1} shows that $K(i_1)$ is a neat subalgebra of
  $(B,\sigma_B)$, which means that $K(i_1)$ is an \'etale $K$-algebra
  in $\Sym(\sigma_B)$ over which $B$ is a free left module. 
  From~\cite[Prop.~6.5]{BGBT1} it follows that
  \[
    \Symd(\sigma_B)=K(i_1)\oplus V \qquad\text{and}\qquad
    \Nrp_{\sigma_B}=N_{K(i_1)/K}\perp \sq_V.
  \]
  The form $N_{K(i_1)/K}$
  is defined over $F$ since $K(i_1)\simeq K\otimes_FF(i_1)$, hence
  $(t_K)_*(N_{K(i_1)/K})$ is hyperbolic and $(t_K)_*(\Nrp_{\sigma_B})$ is
  Witt-equivalent to $(t_K)_*(\sq_V)$.
\end{proof}

To proceed with the proof of Proposition~\ref{prop:sNrp}, fix an
element $i_1\in\Symd(\sigma_B)$ as in Lemma~\ref{lem:insep1}. The
$F$-algebra $K_1:=F(i_1)$ is quadratic \'etale and lies in
$\Symd(\sigma_A)$ since $\Symd(\sigma_B)\subset\Symd(\sigma_A)$. It is
neat in $(A,\sigma_A)$ because $K(i_1)$ is neat in $(B,\sigma_B)$,
$K(i_1)$ is a free left $K_1$-module and $A$ is a free left
$B$-module, hence $A$ is a free left $K_1$-module. By~\cite[Th.~4.1
and 6.10]{BGBT1}, there exists an \'etale biquadratic $F$-algebra
$L\subset\Symd(\sigma_A)$ containing $K_1$ as the subalgebra fixed
under an element $\alpha_1$ of order~$2$ of its Galois group. Let
$\alpha_2$ and $\alpha_3$ be the two other elements of order~$2$ of this Galois group, with fixed subfields respectively denoted by $K_2$ and $K_3$. Consider
\[
  W_i=\{w\in\Symd(\sigma_A)\mid w\ell = \alpha_i(\ell)w \text{ for all
    $\ell\in L$}\}.
\]
Lemma 7.1 and Theorem~7.3 in \cite{BGBT3} or Propositions~4 and 5 in~\cite{CRAS}
show that $w_i^2\in K_i$ for $i\in\{1,2,3\}$ and 
\begin{equation}
  \label{eq:insep1}
  \Symd(\sigma_A) = L\oplus W_1 \oplus W_2\oplus W_3.
\end{equation}
 Moreover, we have 
\begin{multline*} L\oplus W_1=\{w\in\Symd(\sigma_A)\mid w i_1=i_1 w\}\mbox{ and }\\
W_2\oplus W_3= \{w\in\Symd(\sigma_A)\mid w i_1=(i_1+1) w\}. \end{multline*}
Writing $x*y=xy+yx$ for $x$, $y\in A$, we also have 
\[
  w_i*w_j\in W_k \qquad\text{for $w_i\in W_i$, $w_j\in W_j$ and
    $\{i,j,k\} = \{1,2,3\}$}. 
\]

\begin{lemma}
  \label{lem:insep2}
  Let $r\in K$ be such that $t_K(r)=1$. Then $r=k_1+w_1$ for some
  $k_1\in K_1$ and some nonzero $w_1\in W_1$. Moreover, $V\subset
  W_2\oplus W_3$ and for $v=w_2+w_3\in V$ with $w_2\in W_2$ and
  $w_3\in W_3$,
  \begin{equation}
    \label{eq:insep2}
    w_1*w_2= T_{K_1/F}(k_1)w_3, \qquad w_1*w_3 = T_{K_1/F}(k_1)w_2
    \quad\text{and}\quad
    w_2*w_3 = t_K\bigl(\sq_V(v)\bigr)w_1.
  \end{equation}
\end{lemma}

\begin{proof}
  As $\sigma_B$ is symplectic, we have
  $r\in\Symd(\sigma_B)\subset\Symd(\sigma_A)$, hence we may decompose
  $r$ as a sum of terms from the direct sum~\eqref{eq:insep1}. But
  $r$ centralizes $K_1$ since $K_1\subset B$, hence $r\in L\oplus
  W_1$. Let $r=k_1+w_1$ with $k_1\in L$ and $w_1\in W_1$. If $w_1=0$,
  then $r$ lies in $L$. This is impossible since $t_K(r)=1$ implies
  $r\notin F$, hence $r$ is inseparable over $F$ whereas $L$ is
  separable over $F$. Therefore, $w_1\neq0$. As $r^2\in F^\times$, it
  follows that
  \begin{equation}
    \label{eq:insep3}
    k_1^2+w_1^2+k_1*w_1\in F^\times.
  \end{equation}
  We have $k_1^2$, $w_1^2\in L$ and $k_1*w_1\in W_1$,
  hence~\eqref{eq:insep3} implies $k_1*w_1=0$. Since $w_1\neq0$, the
  equations
  \[
    0=k_1*w_1 = k_1w_1+w_1k_1 = \bigl(k_1+\alpha_1(k_1)\bigr)w_1
  \]
  show that $k_1\in K_1$.

  By definition, $V\subset\Symd(\sigma_B)\subset\Symd(\sigma_A)$, and
  $vx = \alpha_1(x)v$ for all $x\in K_1$, hence $V\subset W_2\oplus
  W_3$. For $v\in V$, let $v=w_2+w_3$ for some $w_2\in W_2$ and
  $w_3\in W_3$. Since $v$ commutes with $r=k_1+w_1$, it follows that
  \begin{equation}
    \label{eq:insep4}
    k_1*v+w_1*w_2+w_1*w_3=0.
  \end{equation}
  Note that $k_1*v=k_1v+vk_1 = T_{K_1/F}(k_1)v = T_{K_1/F}(k_1)w_2 +
  T_{K_1/F}(k_1)w_3$. Substituting in~\eqref{eq:insep4} we obtain
  \[
    (T_{K_1/F}(k_1)w_2+w_1*w_3) + (T_{K_1/F}(k_1)w_3 + w_1*w_2) = 0.
  \]
  But $w_1*w_3\in W_2$ and $w_1*w_2\in W_3$, hence the first term on
  the left side lies in~$W_2$ whereas the second term lies in
  $W_3$. Each term must vanish since $W_2\oplus W_3$ is a direct
  sum. This proves the first two equations in~\eqref{eq:insep2}, and
  it only remains to compute $w_2*w_3$.

  For this, observe that $\sq_V(v)\in K$, hence $\sq_V(v) =
  \lambda_0+\lambda_1r$ for some $\lambda_0$, $\lambda_1\in F$. Taking
  images under $t_K$, we obtain $\lambda_1=t_K\bigl(\sq_V(v)\bigr)$, hence
  \[
    \sq_V(v) = (\lambda_0+\lambda_1k_1) + t_K\bigl(\sq_V(v)\bigr)w_1
    \qquad\text{with $\lambda_0+\lambda_1k_1\in L$ and
      $t_K\bigl(\sq_V(v)\bigr)w_1\in W_1$.}
  \]
  Comparing with
  \[
    \sq_V(v) = (w_2+w_3)^2 = w_2^2+w_3^2 + w_2*w_3 \qquad
    \text{with $w_2^2$, $w_3^2\in L$ and $w_2*w_3\in W_1$,}
  \]
  we see that $w_2*w_3 = t_K\bigl(\sq_V(v)\bigr)w_1$. This completes the
  proof of~\eqref{eq:insep2}.
\end{proof}

\begin{proof}[Proof of Proposition~\ref{prop:sNrp}, completed]
  We assume $K$ is inseparable over $F$ and we need to show that
  $(t_K)_*(\Nrp_{\sigma_B})$ is Witt-equivalent to a scalar multiple of
  $\pi(\sigma_A)$. Lemma~\ref{lem:insep1} provides us with a
  quadratic space $(V,\sq_V)$ over $K$ such that
  $(t_K)_*(\Nrp_{\sigma_B})$ is Witt-equivalent to $(t_K)_*(\sq_V)$
  and with 
  $i_1\in\Symd(\sigma_B)$ such that $i_1^2+i_1\in F$. Let $K_1=F(i_1)$
  and choose a biquadratic \'etale $F$-algebra
  $L\subset\Symd(\sigma_A)$ containing $K_1$ as the subalgebra fixed
  under an element $\alpha_1$ of order~$2$ in its Galois group. Using
  the same notation as in Lemma~\ref{lem:insep2}, we
  obtain a direct sum decomposition of $\Symd(\sigma_A)$ as
  in~\eqref{eq:insep1}. Fix $r\in K$ such that $t_K(r)=1$ and
  write $r=k_1+w_1$ with $k_1\in K_1$ and $w_1\in W_1$ as in
  Lemma~\ref{lem:insep2}.
 
  Suppose first $k_1\notin F$, hence $T_{K_1/F}(k_1)\neq0$. For $v\in
  V$, let $v=w_2+w_3$ with $w_2\in W_2$ and $w_3\in W_3$. The
  equations~\eqref{eq:insep2} yield $w_2*w_3=t_K\bigl(\sq_V(v)\bigr)w_1$
  and $w_3=T_{K_1/F}(k_1)^{-1}w_1*w_2$. Eliminating $w_3$ from these
  equations yields
  \begin{equation}
    \label{eq:insep5}
    t_K\bigl(\sq_V(v)\bigr)\,w_1=T_{K_1/F}(k_1)^{-1}w_2*(w_1*w_2).
  \end{equation}
  Note that since $w_2^2\in K_2$,
  \[
    w_2*(w_1*w_2) = w_2w_1w_2+w_1^2w_2 + w_1w_2^2 + w_2w_1w_2 =
    (w_2^2+\alpha_1(w_2^2))w_1 = T_{K_2/F}(w_2^2)w_1.
  \]
  Substituting in~\eqref{eq:insep5} and taking into account that $w_1$
  is nonzero, we obtain
  \[
    t_K\bigl(\sq_V(v)\bigr) = T_{K_1/F}(k_1)^{-1}T_{K_2/F}(w_2^2) =
    T_{K_1/F}(k_1)^{-1}T_{K_2/F}\bigl(\sq_{W_2}(w_2)\bigr),
  \]
  where $\sq_{W_2}\colon W_2\to K_2$ is the squaring map. Therefore,
  the projection $V\to W_2$ parallel to $W_3$, which maps $v$ to
  $w_2$, is a similitude from $(t_K)_*(\sq_V)$ to
  $(T_{K_2/F})_*(\sq_{W_2})$, with multiplier
  $T_{K_1/F}(k_1)^{-1}$. By definition, $\pi(\sigma_A)$ is a scalar
  multiple of $(T_{K_2/F})_*(\sq_{W_2})$. Therefore, $\pi(\sigma_A)$
  is isometric to a scalar multiple of $(t_K)_*(\sq_V)$, which is
  Witt-equivalent to $(t_K)_*(\Nrp_{\sigma_B})$, and the proof is complete
  when $k_1\notin F$.

  For the rest of the proof, assume $k_1\in F$. Then $K=F(r)=F(w_1)$,
  hence $w_1^2\in F^\times$. Since conjugation by $w_1$ yields the
  nontrivial $F$-automorphism of $K_2$, the $F$-algebra $Q\subset A$
  generated by $K_2$ and $w_1$ is a quaternion algebra stable under
  $\sigma_A$ and containing $K$. The restriction of $\sigma_A$ to $Q$
  is not the canonical symplectic involution on $Q$ because
  $K_2\subset\Symd(\sigma_A)$. Therefore, by~\cite[(2.23)]{BoI},
  $\sigma_A$ restricts to a symplectic involution $\sigma_C$ on the
  centralizer $C$ of $Q$ in $A$, which is a biquaternion algebra such
  that $B\simeq C\otimes_F K$ since $A=C\otimes_FQ$. It follows that
  $\Nrp_{\sigma_B} = (\Nrp_{\sigma_C})_K$, hence
  $(t_K)_*(\Nrp_{\sigma_B})$ is hyperbolic. On the other hand, $C$
  decomposes into a tensor product of two quaternion $F$-algebras
  stable under $\sigma_C$ by~\cite[(15.16)]{BoI}, hence $A$ decomposes
  into a tensor product of three quaternion $F$-algebras stable under
  $\sigma_A$. Theorem~7.11 of~\cite{BGBT3} or Proposition~8
  of~\cite{CRAS} then show that $\pi(\sigma_A)$ is hyperbolic. Thus,
  $(t_K)_*(\Nrp_{\sigma_B})$ and $\pi(\sigma_A)$ are Witt-equivalent.
\end{proof}

Proposition~\ref{prop:sNrp} shows that for a given symplectic
involution $\sigma_B$ on $B$, all the involutions $\sigma_A$ on a
central simple $F$-algebra $A$ of degree~$8$ containing $B$ that
extend 
$\sigma_B$ have the same discriminant Pfister form, and that form does
not even depend on the choice of the algebra $A$ but only on
$\sigma_B$. We next describe how the form behaves under a change of
involution on $B$.

\begin{prop}
  \label{prop:deltaKF}
  Let $q_B\in I^2_q(K)$ be such that $e_2(q_B)=[B]$, and let
  $t_K\colon K\to F$ be a nonzero $F$-linear functional such that
  $t_K(1)=0$. If $\sigma_A$ and 
  $\sigma'_A$ are involutions on $A$ that extend symplectic
  involutions $\sigma_B$, $\sigma'_B$ on $B$, then in $\overline
  I^3_q(F)$
  \[
    \overline{\pi(\sigma'_A)} \equiv \overline{\pi(\sigma_A)}
    \quad\bmod\overline{(t_K)_*(\bpf{K^\times}q_B)}.
  \]
\end{prop}

\begin{proof}
  Let $\sigma'_B = \Int(u)\circ\sigma_B$ for some unit
  $u\in\Symd(\sigma_B)$. By~\cite[(16.18)]{BoI} we have
  \[
    \Nrp_{\sigma_B}-\Nrp_{\sigma'_B} = \bpf{\Nrp_{\sigma_B}(u)}
    \Nrp_{\sigma_B} \qquad\text{in $I^2_q(K)$,}
  \]
  hence
  \begin{equation}
    \label{eq:deltaKF}
    (t_K)_*(\Nrp_{\sigma_B}) - (t_K)_*(\Nrp_{\sigma'_B}) =
    (t_K)_*(\bpf{\Nrp_{\sigma_B}(u)}\Nrp_{\sigma_B}) \in
    (t_K)_*(\bpf{K^\times}\Nrp_{\sigma_B}). 
  \end{equation}
  By~\cite[(16.8)]{BoI}, $\Nrp_{\sigma_B}$ is an Albert form of $B$,
  hence $\overline{\Nrp_{\sigma_B}} = \overline{q_B}$
  and~\eqref{eq:deltaKF} readily yields
  \[
    \overline{(t_K)_*(\Nrp_{\sigma_B})} \equiv
    \overline{(t_K)_*(\Nrp_{\sigma'_B})} \quad \bmod
    \overline{(t_K)_*(\bpf{K^\times} q_B)}.
  \]
  This proves the proposition since
  $\overline{(t_K)_*(\Nrp_{\sigma_B})}=\overline{\pi(\sigma_A)}$ and 
  $\overline{(t_K)_*(\Nrp_{\sigma'_B})} = \overline{\pi(\sigma'_A)}$ by
  Proposition~\ref{prop:sNrp}.
\end{proof}
\begin{remark}  An alternative proof can be obtained from
  Theorem~\ref{thm:main} instead of~\cite[(16.18)]{BoI}.
  \end{remark}

We now define an invariant of the biquaternion algebra $B$, which also depends on the quadratic extension $K/F$. 
\begin{defn}
  \label{def:deltaKF}
  Let $B$ be a biquaternion algebra over a quadratic extension $K$ of
  a field $F$, let $q_B\in I^2_q(K)$ be such that
  $e_2(q_B)=[B]$ and let $s_K\colon K\to F$ be a nonzero $F$-linear
  functional. Assume $N_{K/F}[B]=0$. By Theorem~\ref{thm:quadesc} 
  and \cite[(4.14)]{BoI}, there exists a central simple $F$-algebra
  $A$ of degree~$8$ and an involution $\sigma_A$ on $A$ restricting to
  a symplectic involution on $B$. 
  Taking $s_K=t_K$, with $t_K(1)=0$, Propositions~\ref{prop:sNrp} and \ref{prop:deltaKF} show that 
  \[\overline{\pi({\sigma_A})} +
    \overline{(t_K)_*(\bpf{K^\times} q_B)} \in \overline
    I^3_q(F)/\overline{(t_K)_*(\bpf{K^\times} q_B)}\] does not depend on the choices of $A$ and $\sigma_A$. It coincides with the class 
 \[ \overline{(t_K)_*(\Nrp_{\sigma_B})}+
    \overline{(t_K)_*(\bpf{K^\times} q_B)}\in \overline I^3_q(F)/\overline{(t_K)_*(\bpf{K^\times} q_B)}\] for any symplectic involution $\sigma_B$ of $B$. 
    Moreover, let $q_A$ be a quadratic form in $I_q^2(F)$ such that $e_2(q_A)=[A]$. We may take $q_B=(q_A)_K$ and by~\cite[(20.3c) \& (21.5)]{EKM}, we have 
    \[\overline{(s_K)_*(\bpf{K^\times} q_B)}=\overline{(s_K)_*(\bpf{K^\times} )q_A}=\overline{(t_K)_*(\bpf{K^\times} )q_A}=\overline{(t_K)_*(\bpf{K^\times} q_B)}\in\overline I^3_q(F).\]
     We therefore define
  \[
    \delta_{K/F}(B) = \overline{\pi({\sigma_A})} +
    \overline{(s_K)_*(\bpf{K^\times} q_B)} \in \overline
    I^3_q(F)/\overline{(s_K)_*(\bpf{K^\times} q_B)}.
  \]
  
    If $\varphi_B$ is an Albert form of $B$, then $\varphi_B\simeq
  \bil\lambda \Nrp_{\sigma_B}$ for some $\lambda\in K^\times$ because
  all the Albert forms of $B$ are similar
  by~\cite[(16.3) and (16.8)]{BoI}. Hence,
  \[
    (t_K)_*(\varphi_B) = (t_K)_*(\Nrp_{\sigma_B}) +
    (t_K)_*(\bpf\lambda \Nrp_{\sigma_B}),
  \]
  so that $\overline{(t_K)_*(\varphi_{B})}$ and
  $\overline{(t_K)_*(\Nrp_{\sigma_B})}$ have the same image in
  $\overline I^3_q(F)/\overline{(t_K)_*(\bpf{K^\times}q_{B})}$. It
  follows that the definition of $\delta_{K/F}(B)$ can also be stated as
  \[
    \delta_{K/F}(B) = \overline{(t_K)_*(\varphi_B)}+
    \overline{(t_K)_*(\bpf{K^\times} q_B)} \qquad\text{for any
      Albert form  $\varphi_B$ of $B$.}
  \]
  Thus, it is clear that the invariant $\delta_{K/F}(B)$ is the same
  as the one defined in cohomological terms by
  Barry--Chapman--Laghribi in~\cite[\S3]{BaCL} for the case where $K$
  is separable over $F$ and by Laghribi--Lorenz in~\cite[\S6]{LL} for
  the case where $K$ is inseparable.
\end{defn}

Vanishing of $\delta_{K/F}(B)$ characterizes biquaternion $K$-algebras
that descend to biquaternion algebras over $F$:

\begin{prop}
  \label{prop:deltaKFdesc}
  With the notation of Definition~\ref{def:deltaKF}, there exists a
  biquaternion $F$-algebra $B_0$ such that $B\simeq B_0\otimes_FK$ if
  and only if $\delta_{K/F}(B)=0$.
\end{prop}

\begin{proof}
  We sketch the proof given in~\cite[Th.~3.6]{BaCL} or~\cite[Prop.~6.4]{LL}
  for the reader's convenience. If $B\simeq B_0\otimes_FK$, then we
  may choose $\varphi_B=(\varphi_{B_0})_K$ with $\varphi_{B_0}$ an
  Albert form of $B_0$, 
  and then $(t_K)_*(\varphi_B)$ is hyperbolic. Conversely, if
  $\varphi_B$ is an Albert 
  form of $B$ and $\overline{(t_K)_*(\varphi_B)} =
  \overline{(t_K)_*(\bpf\lambda\varphi_B)}$ for some $\lambda\in
  K^\times$, then 
  $\overline{(t_K)_*(\bil\lambda\varphi_B)}=0$ and we may find a
  quadratic form $\psi\in I^2_q(F)$ of dimension~$6$ such that
  $\bil\lambda\varphi_B\simeq (\psi)_K$. Then $\psi$ is an Albert
  form of a 
  biquaternion $F$-algebra $B_0$ such that $B\simeq B_0\otimes_FK$.
\end{proof}

\section{Biquadratic descent of quaternion algebras}
\label{sec:biquadesc}

Throughout this section, $L$ is a biquadratic field extension of the
field $F$ of characteristic~$2$ and $Q$ is a quaternion algebra over
$L$. We assume that $L$ is not purely
inseparable\footnote{See the Appendix for the case where $L$
      is purely inseparable over $F$.} over $F$ and $Q$ 
satisfies the norm conditions in 
Theorem~\ref{thm:biquaddescBr}, so its Brauer class descends to
$F$. More precisely, there exists a central simple $F$-algebra $A$, of degree $8$ and exponent $2$, containing $L$, and such that $Q$ is isomorphic to the centralizer of $L$ in $A$. Our goal is to define an invariant $\delta_{L/F}(Q)$ whose vanishing characterizes quaternion algebras of the form $Q_0\otimes_FL$ for some quaternion $F$-algebra $Q_0$. As in the definitions of the $\Delta$-invariant in Section~\ref{sec:Delta} and of the $\delta_{K/F}$-invariant in Section~\ref{sec:Quadbiquat}, we use discriminant Pfister forms of symplectic involutions on algebras of degree~$8$.
This invariant is closely related to the one defined by Laghribi--Lorenz in~\cite{LL}, as we also show. Its values lie in
the quotient $\overline
I^3_q(F)/\overline{(s_L)_*(\bpf{L^\times}\Nrd_Q)}$, where $\Nrd_Q$ is the
reduced norm quadratic form and $s_L\colon L\to F$ is any nonzero
$F$-linear functional. 
Using a quadratic form $q_A\in I^2_q(F)$ with $e_2(q_A)=[A]$, the same kind of computations as in Definition~\ref{def:deltaKF} and ~\cite[Ex.~21.6]{EKM}
show that this quotient does not depend on the choice of $s_L$.

The definition of $\delta_{L/F}(Q)$ follows the same lines as for the
invariant of biquaternion algebras in Section~\ref{sec:Quadbiquat}. By~\cite[(4.14)]{BoI}, the canonical (symplectic)
involution $\gamma$ on $Q$ extends to an involution $\sigma_A$ on $A$,
which is necessarily symplectic. We show in
Proposition~\ref{prop:defdeltaLF} below that the
discriminant Pfister form $\pi(\sigma_A)$ induces a class in $\overline
I^3_q(F)/\overline{(s_L)_*(\bpf{L^\times}\Nrd_Q)}$ which does not depend on the choices of the algebra $A$ and the involution $\sigma_A$. 
We define this class to be the invariant $\delta_{L/F}(Q)$, see Definition~\ref{def:deltaLFQ} below.

\medbreak

We first consider a change of algebra $A$ containing $Q$:

\begin{lemma}
  \label{lem:chgA}
  Let $A$ and $A'$ be central simple $F$-algebras of degree~$8$
  containing $Q$ such that $2[A]=2[A']=0$, and let $\sigma_A$ be an
  involution on $A$ extending the canonical involution $\gamma$ on
  $Q$. There exists an involution $\sigma_{A'}$ on $A'$ extending
  $\gamma$ on $Q$ such that $\pi(\sigma_{A'})\simeq
  \pi(\sigma_A)$. 
\end{lemma}

\begin{proof}
  Let $K_1$ and $K_2$ be quadratic extensions of $F$ such that
  $L\simeq K_1\otimes_FK_2$. The Brauer class $[A']-[A]\in \Br(F)$ is
  split by $L$ since $[A]_L=[Q]=[A']_L$. Therefore, $[A']-[A]$ is the
  Brauer class of a biquaternion $F$-algebra $B$. Let $q_B$ be an
  Albert form of $B$. Since $B$ is split by $L$, the form $(q_B)_L$ is
  hyperbolic. Therefore, by~\cite[Th.~2.1]{Ahmad} if one of $K_1$,
  $K_2$ is inseparable over $F$, or by~\cite[Th.~2.2]{Ahmad} if $K_1$
  and $K_2$ are separable over $F$, the form $q_B$ is Witt-equivalent
  to an orthogonal sum $\varphi_1\perp\varphi_2$, where $\varphi_i$ is
  split by $K_i$ for $i=1$, $2$. Comparing Arf invariants shows that
  $\varphi_1$, $\varphi_2\in I^2_q(F)$. Letting $e_2(\varphi_i)=[Q_i]$
  we have
  \[
    [A'] = [A] + [Q_1] + [Q_2] \qquad\text{with $Q_i$ a quaternion
      algebra split by $K_i$ for $i=1$, $2$.}
  \]

  Let $B_1\subset A$ be the centralizer of $K_1$; thus $B_1\supset Q$
  and, since $Q_1$ is split by $K_1$,
  \[
    [B_1]=[A]_{K_1} = [A\otimes_FQ_1]_{K_1}.
  \]
  Every field extension of $K_1$ that splits $A$ also splits
  $A\otimes_FQ_1$, hence $A\otimes_FQ_1$ is Brauer-equivalent to a
  central simple 
  $F$-algebra $A_1$ of degree~$8$, and we
  may assume $A_1$ contains $B_1$ since $[A_1]_{K_1}=[B_1]$. Let
  $\sigma_{B_1}$ be the restriction of $\sigma_A$ to
  $B_1$. By~\cite[(4.14)]{BoI} there exists an involution
  $\sigma_{A_1}$ on $A_1$ extending $\sigma_{B_1}$. Let $s_1\colon
  K_1\to F$ be a nonzero $F$-linear functional such that
  $s_1(1)=0$. Proposition~\ref{prop:sNrp} shows that $\pi(\sigma_A)$
  and $\pi(\sigma_{A_1})$ are both Witt-equivalent to a scalar multiple of
  $(s_1)_*(\Nrp_{\sigma_{B_1}})$, hence
  $\pi(\sigma_A)\simeq\pi(\sigma_{A_1})$.

  Now, $[A']=[A_1]+[Q_2]$, hence $[A']_{K_2} = [A_1]_{K_2}$ since
  $K_2$ splits $Q_2$. Therefore, up to isomorphism we may assume that
  the centralizer $B_2$ of $K_2$ in $A_1$ is contained in $A'$. Let
  $\sigma_{B_2}$ be the restriction of $\sigma_{A_1}$ to
  $B_2$. By~\cite[(4.14)]{BoI} again, there exists an involution
  $\sigma_{A'}$ on $A'$ extending $\sigma_{B_2}$, and the same
  argument as for $\sigma_{A_1}$ above using
  Proposition~\ref{prop:sNrp} shows that $\pi(\sigma_{A_1})\simeq
  \pi(\sigma_{A'})$, hence also $\pi(\sigma_{A'})\simeq
  \pi(\sigma_A)$. 
  Note that the involutions $\sigma_A$, $\sigma_{B_1}$, $\sigma_{A_1}$, $\sigma_{B_2}$ and $\sigma_{A'}$ all coincide on $Q$, and therefore extend its canonical involution $\gamma$.  
  Thus, $\sigma_{A'}$ meets the requirements of the lemma.
\end{proof}

We also require the following technical result:

\begin{lemma}
  \label{lem:tech}
  Let $A$ be a central simple $F$-algebra of degree~$8$ containing $Q$
  and $\sigma_A$ be an involution on $A$ extending $\gamma$ on
  $Q$. Let also $q_A\in I^2_q(F)$ be such that $e_2(q_A)=[A]$, and let
  $s_L\colon L\to F$ be a nonzero $F$-linear functional. For every
  $\ell\in L^\times$,
  \begin{equation}
    \label{eq:tech}
    N_{L/F}(\ell) = \Nrp_{\sigma_A}(\ell) \quad\text{and}\quad
    \overline{(s_L)_*(\bpf\ell \Nrd_Q)} = \overline{\bpf{N_{L/F}(\ell)}
      q_A} \quad\text{in $\overline I^3_q(F)$.}
  \end{equation}
\end{lemma}

\begin{proof}
  First, observe that $L=\Symd(\gamma)\subset\Symd(\sigma_A)$, hence
  $\Nrp_{\sigma_A}(\ell)$ is defined for all $\ell\in L$.
  Since $L$ is a subfield of degree~$4$ in $A$ of degree~$8$, we have
  $\Nrd_A(\ell)=N_{L/F}(\ell)^2$, hence $\Nrp_{\sigma_A}(\ell) =
  N_{L/F}(\ell)$ since $\Nrd_A(\ell)=\Nrp_{\sigma_A}(\ell)^2$.

  Next, note that $e_2(\Nrd_Q) = [Q]=[A_L]$, hence
  $\overline{\Nrd_Q} = \overline{(q_A)_L}$ in $\overline I^2_q(L)$. By
  Frobenius reciprocity (see~\cite[Prop.~20.2]{EKM}), it follows that
  $\overline{(s_L)_*(\bpf \ell \Nrd_Q)} = \overline{(s_L)_*(\bpf
    \ell)q_A}$. 
  If $L=F(\ell)$, then 
  \cite[Cor.~20.13 and 21.5]{EKM} show that $\overline{(s_L)_*(\bpf
    \ell)} = \overline{\bpf{N_{L/F}(\ell)}}$ in $\overline I(F)$,
  and~\eqref{eq:tech} is proven. If $F(\ell)\subsetneq L$, then
  $N_{L/F}(\ell)\in F^{\times2}$, hence $\bpf{N_{L/F}(\ell)} =
  0$. On the other hand, taking $s_L=s_1\circ s_2$ where $s_2\colon L\to
  F(\ell)$ is a nonzero $F(\ell)$-linear functional that vanishes on
  $F(\ell)$ and $s_1\colon F(\ell)\to F$ is a nonzero $F$-linear
  functional, we 
  obtain $(s_L)_*(\bpf \ell)=0$. Therefore, \eqref{eq:tech} holds in all
  cases. 
\end{proof}

\begin{prop}
  \label{prop:defdeltaLF}
  Let $A$ and $A'$ be central simple $F$-algebras of degree~$8$
  containing $Q$ and $\sigma_A$, $\sigma_{A'}$ be involutions on $A$
  and $A'$ respectively extending the canonical involution $\gamma$ on
  $Q$. Then, for any nonzero $F$-linear functional $s_L\colon L\to F$,
  \[
    \overline{\pi(\sigma_A)} \equiv \overline{\pi(\sigma_{A'})}
    \bmod \overline{(s_L)_*(\bpf{L^\times}\Nrd_Q)} \quad\text{in
      $\overline I^3_q(F)$}.
  \]
\end{prop}

\begin{proof}
  Lemma~\ref{lem:chgA} yields an involution $\sigma'_{A'}$ on $A'$
  extending $\gamma$ such that $\pi(\sigma'_{A'})\simeq
  \pi(\sigma_A)$. Let $\sigma'_{A'}=\Int(\ell)\circ\sigma_{A'}$ for
  some $\ell\in\Symd(\sigma_{A'})$. Since the restrictions of
  $\sigma'_{A'}$ and $\sigma_{A'}$ to $Q$ are the same, $\ell$
  centralizes $Q$ hence $\ell\in L^\times$. Theorem~\ref{thm:main}
  yields 
  \[
    \overline{\pi(\sigma'_{A'})} =
    \overline{\bpf{\Nrp_{\sigma_{A'}}(\ell)} q_{A'}} +
    \overline{\pi(\sigma_{A'})} \qquad\text{in $\overline 
      I^3_q(F)$},
  \]
  where $q_{A'}\in I^2_q(F)$ is any quadratic form such that
  $e_2(q_{A'})=[A']$
  and Lemma~\ref{lem:tech} shows that
  \[
    \overline{\bpf{\Nrp_{\sigma_{A'}}(\ell)} q_{A'}} =
    \overline{(s_L)_*(\bpf\ell \Nrd_Q)} \quad\text{in
      $\overline I^3_q(F)$}.
  \]
  Therefore, $\overline{\pi(\sigma_{A})} =
  \overline{\pi(\sigma'_{A'})} \equiv \overline{\pi(\sigma_{A'})}
  \bmod \overline{(s_L)_*(\bpf{L^\times}\Nrd_Q)}$.
\end{proof}

Proposition~\ref{prop:defdeltaLF} allows us to set the following
definition:

\begin{defn}
  \label{def:deltaLFQ}
  Let $Q$ be a quaternion algebra over a biquadratic field extension
  $L$ of $F$. Assume $L$ is not purely inseparable over $F$ and
  $N_{L/K}[Q]=0$ for all quadratic extensions $K$ of $F$ in $L$, so
  that there exist central simple $F$-algebras $A$ of degree~$8$
  containing $Q$ such that $2[A]=0$ by
  Theorem~\ref{thm:biquaddescBr}. An invariant of $Q$ is defined
  by
  \[
    \delta_{L/F}(Q) = \overline{\pi(\sigma_A)} +
    \overline{(s_L)_*(\bpf{L^\times}\Nrd_Q)}
    \in
    \overline I^3_q(F)/\overline{(s_L)_*(\bpf{L^\times}\Nrd_Q)}
  \]
  for any involution $\sigma_A$ on $A$ extending the canonical
  involution $\gamma$ on $Q$ and any nonzero $F$-linear functional
  $s_L\colon L\to F$.
\end{defn}

Vanishing of the invariant $\delta_{L/F}(Q)$ characterizes quaternion
$L$-algebras that descend to $F$:

\begin{thm}
  \label{thm:biquadesc}
  For a quaternion algebra $Q$ as in Definition~\ref{def:deltaLFQ},
  $\delta_{L/F}(Q)=0$ if and only if there exists a quaternion
  $F$-algebra $Q_0$ such that $Q\simeq Q_0\otimes_FL$.
\end{thm}

\begin{proof}
  Assume $Q=Q_0\otimes_FL$ for some quaternion $F$-algebra $Q_0$, and
  let $L\simeq K_1\otimes_FK_2$ for some quadratic extensions $K_1$,
  $K_2$ of $F$. Each $\End_F(K_i)$ is a split quaternion $F$-algebra,
  and a canonical embedding $K_i\hookrightarrow \End_F(K_i)$ is given
  by left multiplication. We may therefore take
  $A=Q_0\otimes_F\End_F(K_1)\otimes_F\End_F(K_2)$ in
  Definition~\ref{def:deltaLFQ}, since we may consider $Q$ as a
  subalgebra of $A$. For $i=1$, $2$, let $\theta_i$ be an involution
  on $\End_F(K_i)$ leaving $K_i$ fixed. The involution
  $\sigma_A=\gamma_0\otimes\theta_1\otimes\theta_2$ on $A$, where
  $\gamma_0$ is the canonical involution on $Q_0$, extends the
  canonical involution $\gamma$ on $Q$. Since $\sigma_A$ is decomposable,
  $\pi(\sigma_A)$ is hyperbolic by~\cite[Th.~7.11]{BGBT3} or
  \cite[Prop.~8]{CRAS}, so $\delta_{L/F} (Q)=0$ as required. 

  For the converse, assume $\delta_{L/F}(Q)=0$ and consider an
  involution $\sigma_A$ extending $\gamma$ on a central simple
  $F$-algebra $A$ of degree~$8$ containing $Q$. By definition, the
  equation $\delta_{L/F}(Q)=0$ means that
  \[
    \overline{\pi(\sigma_A)} = \overline{(s_L)_*(\bpf\ell \Nrd_Q)}\,\in \overline I^3_q(F)
    \qquad\text{for some $\ell\in L^\times$.}
  \]
  Lemma~\ref{lem:tech} allows us to reformulate this equation as
  \[
    \overline{\pi(\sigma_A)} = \overline{\bpf{\Nrp_{\sigma_A}(\ell)}q_A}\,\in \overline I^3_q(F)
  \]
  for any quadratic form $q_A\in I^2_q(F)$ such that
  $e_2(q_A)=[A]$. For $\sigma'_A = \Int(\ell)\circ\sigma_A$ we then
  have $\overline{\pi(\sigma'_A)}=0$ in $\overline I^3_q(F)$ by
  Theorem~\ref{thm:main}. Since $\pi(\sigma'_A)$ is a $3$-fold
  Pfister form, this implies that $\pi(\sigma'_A)$ is
  hyperbolic. 
  Note that $\ell\in L^\times$ commutes with $Q$, so that $\sigma'_A$ also extends the canonical involution $\gamma$ of $Q$. 
 
  If $L$ is separable over $F$,
  it follows from~\cite[Prop.~7.6]{BGBT3} that $\sigma'_A$ decomposes
  ``along $L$'', which implies that $A\simeq Q_0\otimes_FB$ for some
  quaternion $F$-algebra $Q_0$ and some biquaternion $F$-algebra $B$
  containing $L$. As $Q$ is the centralizer of $L$ in $A$, we then
  have $Q\simeq Q_0\otimes_FL$, and the proof is complete when $L$ is
  separable over $F$.

  For the rest of the proof, assume $L=K_1\otimes_FK_2$ where $K_1$
  (resp.~$K_2$) is a separable (resp.~inseparable) quadratic extension
  of $F$. Let $B_2\subset A$ be the centralizer of $K_2$ in $A$ and
  $\sigma_{B_2}=\sigma'_A\rvert_{B_2}$ be the restriction of
  $\sigma'_A$ 
  to $B_2$. Let also $i_1\in K_1$ be such that $T_{K_1/F}(i_1)=1$, and
  consider
  \[
    V=\{v\in\Symd(\sigma_{B_2})\mid vi_1=(i_1+1)v\}.
  \]
  Lemma~\ref{lem:insep1} shows that the squaring map $\sq_V\colon v\to
  v^2$ is a nonsingular quadratic form $V\to K_2$, whose transfer
  $(t_{K_2})_*(\sq_V)$ for any $F$-linear functional $t_{K_2}\colon K_2\to F$
  such that $t_{K_2}(1)=0$ is Witt-equivalent to
  $(t_{K_2})_*(\Nrp_{\sigma_{B_2}})$. It is therefore also Witt-equivalent
  to $\pi(\sigma'_A)$ by Proposition~\ref{prop:sNrp}, hence it is
  hyperbolic. Theorem~2.2 in \cite{BGBT2} then shows that $\sq_V
  \simeq \varphi_{K_2}$ for some nonsingular quadratic form
  $\varphi\in I_q(F)$. In particular, we may find $v\in V$ such that
  $v^2\in F^\times$. Then $i_1$ and $v$ generate a quaternion
  $F$-algebra $Q_1\subset B_2$ containing $K_1$, and $B_2\simeq
  Q_1\otimes_FQ_2$ for $Q_2$ the centralizer of $Q_1$ in $B_2$. As
  $[B_2]=[A]_{K_2}$, Theorem~\ref{thm:quadesc} shows that
  $N_{K_2/F}[B_2]=0$, hence $N_{K_2/F}[Q_2]=0$ since
  $N_{K_2/F}[(Q_1)_{K_2}]=0$. Therefore,
  Proposition~\ref{prop:descquat} yields a quaternion $F$-algebra
  $Q_0$ such that $Q_2\simeq Q_0\otimes_FK_2$. Thus, $B_2\simeq
  Q_1\otimes_FQ_0\otimes_FK_2$. 
  The algebra $Q$ is the centralizer of $K_1\otimes_F K_2$ in $A$, hence also in $B_2$. 
  Since $K_1\subset Q_1$, it follows that $Q\simeq
  Q_0\otimes_FK_1\otimes_FK_2=Q_0\otimes_FL$. 
\end{proof}

We complete this section with an explicit computation of the invariant
$\delta_{L/F}(Q)$, which shows that this invariant is the same as the
one defined by Laghribi--Lorenz in~\cite{LL} (even though 
they view it as an invariant of a central simple
$F$-algebra $A$ of degree~$8$ containing $Q$).

\begin{prop}
  \label{prop:exdeltaLF}
  Let $Q$ be a quaternion algebra over a biquadratic field extension
  $L$ of $F$. Assume $L$ is not purely inseparable over $F$ and
  let $L\simeq F\bigl(\wp^{-1}(\alpha)\bigr)\otimes_FK$ for some
  $\alpha\in F$ and some quadratic field extension $K$ of $F$
  (separable or not). Assume moreover that $Q$ satisfies the
  conditions in Theorem~\ref{thm:biquaddescBr}; then
  \[
    Q\simeq[\alpha_0,\beta_0)_L \qquad\text{for some $\alpha_0\in F$
      and $\beta_0\in K^\times$,}
  \]
  and there exists $\beta\in K^\times$ such that
  \begin{equation}
    \label{eq:NKF2}
    [\alpha_0,N_{K/F}(\beta_0))_F \simeq [\alpha,N_{K/F}(\beta))_F.
  \end{equation}
  For any such $\beta\in K^\times$, the invariant $\delta_{L/F}(Q)$
  vanishes if $\beta\in F$ or $\beta_0\beta\in F$; if $\beta\notin F$
  and $\beta_0\beta\notin F$, 
  then $\delta_{L/F}(Q)$ is represented in $\overline
  I^3_q(F)/\overline{s_*(\bpf{L^\times}\Nrd_Q)}$ by the 
  $3$-fold Pfister form
  \[
    \qpf{t_K(\beta)t_K(\beta_0\beta),
      N_{K/F}(\beta_0\beta),\alpha_0},
  \]
  where $t_K\colon K\to F$ is any nonzero $F$-linear functional such
  that $t_K(1)=0$.
\end{prop}

\begin{proof}
  Fix some central simple $F$-algebra $A$ of degree~$8$ containing $Q$
  such that $2[A]=0$, and let $B$ be the centralizer of $K$ in $A$.
  Since $N_{L/K}[Q]=0$, Proposition~\ref{prop:descquat} yields a
  quaternion $K$-algebra $Q_0$ such that $Q\simeq Q_0\otimes_KL$. The
  centralizer of $Q_0$ in $B$ is therefore a quaternion $K$-algebra
  $Q_1$ containing $L$. Thus, $B=Q_0\otimes_KQ_1$ and
  $Q_1=[\alpha,\beta)_K$ for some $\beta\in K^\times$. Since $B\subset
  A$, Theorem~\ref{thm:quadesc} shows that $N_{K/F}[B]=0$, hence
  \begin{equation}
    \label{eq:NKF1}
    N_{K/F}[Q_0] = N_{K/F}[Q_1] = [[\alpha, N_{K/F}(\beta))_F].
  \end{equation}
  These equations show that the index of $N_{K/F}[Q_0]$ is at
  most~$2$, hence by~\cite[Th.~1.1]{BGBT2} $Q_0$ is split by some
  quadratic \'etale $K$-algebra of the form
  $K\bigl(\wp^{-1}(\alpha_0)\bigr)$ with $\alpha_0\in F$. We may then
  find $\beta_0\in K^\times$ such that 
  $Q_0\simeq[\alpha_0,\beta_0)_K$. It follows that
  $Q\simeq[\alpha_0,\beta_0)_L$, and~\eqref{eq:NKF1}
  yields~\eqref{eq:NKF2}.

  To complete the proof, we calculate $\delta_{L/F}(Q)$. Let $\beta\in
  K^\times$ be such that~\eqref{eq:NKF2} holds, and let
  \[
    B'=[\alpha_0,\beta_0)_K\otimes_K[\alpha,\beta)_K.
  \]
  Since $Q\simeq
  [\alpha_0,\beta_0)_K\otimes_KK(\wp^{-1}(\alpha))$, we may identify
  $Q$ with a subalgebra of $B'$. From~\eqref{eq:NKF2}, it follows that
  $N_{K/F}[B']=0$, hence by Theorem~\ref{thm:quadesc} there exists a
  central simple $F$-algebra $A'$ of degree~$8$ containing $B'$, hence
  also $Q$, such that $2[A']=0$.
  
  If $\beta\in
  F^\times$, then $N_{K/F}(\beta)\in F^{\times2}$ hence
  \eqref{eq:NKF2} shows that
  $[\alpha_0,N_{K/F}(\beta_0))_F$ is split. By
  Proposition~\ref{prop:descquat} it follows that $Q_0\simeq
  Q'_0\otimes_FK$ for some quaternion $F$-algebra $Q'_0$. Then
  $Q\simeq Q'_0\otimes_FL$ and Theorem~\ref{thm:biquadesc} yields
  $\delta_{L/F}(Q)=0$.

  Similarly, if $\beta_0\beta\in F^\times$, then $N_{K/F}(\beta_0)
  \equiv 
  N_{K/F}(\beta)\bmod F^{\times2}$, hence \eqref{eq:NKF2} shows that
  $[\alpha_0+\alpha,N_{K/F}(\beta_0))_F$ is split. By
  Proposition~\ref{prop:descquat} we may then find a quaternion
  $F$-algebra $Q''_0$ such that $[\alpha_0+\alpha,\beta_0)_K \simeq
  Q''_0\otimes_FK$. But since $\alpha\in\wp(L)$ we have
  $Q_0\simeq[\alpha_0+\alpha, \beta_0)_L$, hence $Q\simeq
  Q''_0\otimes_FL$ and $\delta_{L/F}(Q)=0$ by
  Theorem~\ref{thm:biquadesc}.

  For the rest of the proof, assume $\beta\notin F$ and
  $\beta_0\beta\notin F$, and choose a nonzero $F$-linear functional
  $t_K\colon K\to F$ such that $t_K(1)=0$, hence
  $t_K(\beta_0\beta)t_K(\beta)\neq0$. Let $(1,i_0,j_0,i_0j_0)$ and
  $(1,i,j,ij)$ be the quaternion bases of $[\alpha_0,\beta_0)_K$ and
  $[\alpha,\beta)_K$ respectively, such that
  \[
    i_0^2+i_0=\alpha_0,\quad j_0i_0=(i_0+1)j_0,\quad
    j_0^2=\beta_0\quad\text{and} \quad 
    i^2+i=\alpha,\quad ji=(i+1)j,\quad j^2=\beta.
  \]
  Let $\gamma_0$ and $\gamma$ be the canonical involutions on
  $[\alpha_0,\beta_0)_K$ and $[\alpha,\beta)_K$ respectively, and
  consider the involution $\sigma_{B'}$ on $B'$ given by
  \[
    \sigma_{B'} = \gamma_0\otimes(\Int(j)\circ\gamma).
  \]
  This involution $\sigma_{B'}$ extends the canonical involution on
  $Q$, and \cite[(16.15)]{BoI} yields
  \[
    \Nrp_{\sigma_{B'}}\simeq \bil\beta \bigl([1,\alpha_0+\alpha] \perp
    \bil{\beta_0}[1,\alpha_0] \perp \bil\beta [1,\alpha]\bigr),
  \]
  hence by Frobenius reciprocity (see~\cite[Prop.~20.2]{EKM})
  $(t_K)_*(\Nrp_{\sigma_{B'}})$ is Witt-equivalent to
  \begin{multline}
    \label{eq:NKF3}
    (t_K)_*(\bil\beta)[1,\alpha_0+\alpha] \perp
    (t_K)_*(\bil{\beta_0\beta})[1,\alpha_0] \simeq \\
    \bil{t_K(\beta)}\qpf{N_{K/F}(\beta),\alpha_0+\alpha} \perp
    \bil{t_K(\beta_0\beta)}\qpf{N_{K/F}(\beta_0\beta),\alpha_0}. 
  \end{multline}
  Observe that~\eqref{eq:NKF2} implies $[\alpha_0+\alpha,
  N_{K/F}(\beta_0\beta))_F \simeq [\alpha_0,N_{K/F}(\beta_0\beta))_F$,
  hence
  \[
    \qpf{N_{K/F}(\beta),\alpha_0+\alpha} \simeq
    \qpf{N_{K/F}(\beta_0\beta), \alpha_0}
  \]
  and we may rewrite the right side of~\eqref{eq:NKF3} as
  \[
    \bil{t_K(\beta), t_K(\beta_0\beta)}
    \qpf{N_{K/F}(\beta_0\beta),\alpha_0}.
  \]
  Therefore, 
  $
    \overline{(t_K)_*(\Nrp_{\sigma_{B'}})} =
    \overline{\qpf{t_K(\beta)t_K(\beta_0\beta),
        N_{K/F}(\beta_0\beta),\alpha_0}}
  $ in $\overline I^3_q(F)$.
  For any involution $\sigma_{A'}$ on $A'$ extending $\sigma_{B'}$ we
  have 
  $\overline{\pi(\sigma_{A'})} = \overline{(t_K)_*(\Nrp_{\sigma_{B'}})}$
  by Proposition~\ref{prop:sNrp}. Since by definition
  $\delta_{L/F}(Q)$ 
  is represented by $\overline{\pi(\sigma_{A'})}$, the proof is
  complete. 
\end{proof}

\section{Nontriviality of the invariants}
\label{sec:nontriv}

The purpose of this section is to give examples where the various
invariants introduced in Sections~\ref{sec:Delta},
\ref{sec:Quadbiquat} and \ref{sec:biquadesc} do not vanish.

Let $L$ be a biquadratic field extension of $F$ contained in a central
simple $F$-algebra $A$ of degree~$8$ such that $2[A]=0$. Assume $L$ is
not purely inseparable over $F$, and let
$K\subset L$ be a quadratic extension of $F$ in $L$. Let also $Q$ and
$B$ be the centralizers of $L$ and $K$ in $A$ respectively, so 
\[
  F\subset K\subset L\subset Q\subset B\subset A.
\]
Fix nonzero $F$-linear functionals $s_L\colon L\to F$ and
$s_K\colon K\to F$, and let $q_B\in 
I^2_q(K)$, $q_A\in I_q^2(F)$ be such that $e_2(q_B)=[B]$ and
$e_2(q_A)=[A]$. As observed in Lemma~\ref{lem:tech}, we have
\[
  \overline{(s_L)_*(\bpf{L^\times}\Nrd_Q)} =
  \overline{\bpf{N_{L/F}(L^\times)}q_A}.
\]
Similarly, since $\overline{(q_A)_K}=\overline{q_B}$, Frobenius
reciprocity yields
\[
  \overline{(s_K)_*(\bpf{K^\times}q_B)} =
  \overline{\bpf{N_{K/F}(K^\times)}q_A}.
\]
Therefore, there are canonical maps
\begin{equation}
  \label{eq:canmaps}
  \overline I_q^3(F)/\overline{(s_L)_*(\bpf{L^\times}\Nrd_Q)}
  \xrightarrow{\Phi}
    \overline I_q^3(F)/\overline{(s_K)_*(\bpf{K^\times}q_B)}
    \xrightarrow{\Psi}
    \overline I_q^3(F)/\overline{\bpf{F^\times}q_A}
\end{equation}
induced by the inclusions
$\overline{\bpf{N_{L/F}(L^\times)}q_A}\subset
\overline{\bpf{N_{K/F}(K^\times)}q_A}\subset
\overline{\bpf{F^\times}q_A}$.

If $\sigma_A$ is an involution on $A$ extending the canonical
involution on $Q$, then $\pi(\sigma_A)$ represents
 the invariant  $\Delta(A)\in
\overline I_q^3(F)/\overline{\bpf{F^\times}q_A}$, but also 
$\delta_{L/F}(Q)\in \overline
I_q^3(F)/\overline{(s_L)_*(\bpf{L^\times}\Nrd_Q)}$, and
$\delta_{K/F}(B) \in \overline
I_q^3(F)/\overline{(s_K)_*(\bpf{K^\times}q_B)}$, 
see Definitions~\ref{def:Delta},~\ref{def:deltaKF} and \ref{def:deltaLFQ}. Therefore,
\begin{equation}
  \label{eq:Gamma}
  \Phi(\delta_{L/F}(Q)) = \delta_{K/F}(B)
  \qquad\text{and}\qquad
  \Psi(\delta_{K/F}(B)) = \Delta(A).
\end{equation}
Thus, any example of division algebra $A$ such that $\Delta(A)\neq0$
readily yields examples of biquaternion algebras $B$ and quaternion
algebras $Q$ such that $\delta_{K/F}(B)\neq0$ and
$\delta_{L/F}(Q)\neq0$. The converse does not hold:
Remark~\ref{rem:deltaKFtriv} yields an example where
$\delta_{L/F}(Q)\neq0$ and $\delta_{K/F}(B)=0$, and
Remark~\ref{rem:Deltatriv} an example where $\delta_{K/F}(B)\neq0$
and $\Delta(A)=0$.
\medbreak

We discuss below two types of construction: first, we show how to
obtain examples where $\Delta(A)\neq0$ (hence $\delta_{K/F}(B)\neq0$
and $\delta_{L/F}(Q)\neq0$) from any indecomposable division algebra
of degree~$8$ and exponent~$2$. In these examples, the biquadratic
extension $L/F$ is separable. Then, we take the opposite approach:
we start by giving an
example of a quaternion algebra $Q$ with $\delta_{L/F}(Q)\neq0$, and
build from this a biquaternion algebra $B$ such that
$\delta_{K/F}(B)\neq0$ and a central simple $F$-algebra $A$
such that $\Delta(A)\neq0$. In contrast
with the construction sketched above, the biquadratic field extension
$L/F$ and the quadratic field extension $K/F$ are not
separable.

In each case, the main
tools for the proofs are residue maps for discrete valuations defined
by Arason~\cite{Ara2018} and used by
Barry--Chapman--Laghribi~\cite{BaCL} and by
Laghribi--Lorenz~\cite{LL}. 

We briefly recall the definitions from~\cite{Ara2018} and
  \cite[\S6]{BaCL}: if $E$ is a field of characteristic~$2$ with a
  discrete valuation $v$ and uniformizing element $\pi$, the
  \emph{tame} subgroup $I_{q}(E)_0\subset I_q(E)$ is the
  $W(E)$-submodule generated by elements of the form $[1,u]$ with
  $v(u)\geq0$. Residue maps
  \[
    \partial_{q0}\colon I_{q}(E)_0 \to I_q(\overline E)
    \qquad\text{and}\qquad
    \partial_{q1}\colon I_{q}(E)_0 \to I_q(\overline E),
  \]
  where $\overline E$ is the residue field of $E$, are defined as
  follows: 
  for $u_1$, $u_2\in E$ with $v(u_1)=0$ and $v(u_2)\geq0$,
  \begin{align*}
    \partial_{q0}(\bil{u_1}[1,u_2])
    &= \bil{\overline{u_1}}[1,\overline{u_2}],
                                     \qquad&
    \partial_{q0}(\bil{u_1\pi}[1,u_2]) &=0
    \\
    \text{and}\quad
    \partial_{q1}(\bil{u_1}[1,u_2])&=0, \qquad&
      \partial_{q1}(\bil{u_1\pi}[1,u_2]) &=                               
    \bil{\overline{u_1}}[1,\overline{u_2}].
  \end{align*}
  Thus, the residue maps $\partial_{q0}$ and $\partial_{q1}$ are
  related to the residue homomorphisms $\partial_0$, $\partial_1\colon
  W(E)\to W(\overline E)$ of~\cite[\S19.A]{EKM} as follows: for $i=0$,
  $1$,
  \[
    \partial_{qi}(\varphi\cdot[1,u]) = \partial_i(\varphi)\cdot
    [1,\overline u] \qquad
    \text{for $\varphi\in W(E)$ and $u\in E^\times$ with $v(u)\geq0$.}
  \]
  For $n\geq2$, let
  $I^n_{q}(E)_0 = 
  I^{n-1}(E)\cdot I_{q}(E)_0$. Arason's theorem in~\cite{Ara2022}
  shows that $I_q^n(E)_0 = I^n_q(E)\cap I_q(E)_0$. Since
  $\partial_i\bigl(I^n(E)\bigr) \subset I^{n-1}(\overline E)$ for
  $i=0$, $1$ and $n\geq1$ (see~\cite[Lemma~19.14]{EKM}), it follows
  that $\partial_{qi}(I_q^n(E)_0) \subset I^{n-1}_q(\overline E)$ for
  $i=0$, $1$ and $n\geq2$. We also use the map
  \[
    \partial_{q+}:=\partial_{q0} + \partial_{q1}\colon I_q(E)_0 \to
    I_q(\overline E)
  \]
  such that for $u_1$, $u_2\in E^\times$ with $v(u_1)=0$ and
  $v(u_2)\geq0$ 
  \[
    \partial_{q+}(\bil{u_1}[1,u_2])=\partial_{q+}(\bil{u_1\pi}[1,u_2]) =
    \bil{\overline{u_1}}[1,\overline{u_2}].
  \]
  The calculations in~\cite[Ex.~19.13]{EKM} show that
  $(\partial_0+\partial_1)\bigl(I^n(E)\bigr) \subset I^n(\overline E)$
  for $n\geq0$, hence $\partial_{q+}(I^n_{q}(E)_0\bigr) \subset
  I^n_q(\overline E)$.
\subsection*{First construction}
  
Let $A$ be a central simple $F$-algebra of degree~$8$ such that
$2[A]=0$. If its index is~$2$ or $4$, then clearly $A$ contains a
separable quadratic field extension of $F$ since it is
Brauer-equivalent to a quaternion algebra or to a biquaternion
algebra. If $A$ is a division algebra, then by a theorem of Rowen
(see~\cite[Th.~5.6.10]{J}) it contains a separable triquadratic extension
of $F$. Thus, assuming $A$ is not split we may fix a separable
quadratic field extension $K=F(\wp^{-1}(a))$ of $F$ in $A$ and
consider its centralizer $B\subset A$. Let $t$ be an indeterminate
over $F$. The $F(t)$-algebra $A\otimes_F[a,t)_{F(t)}$ is split by
every extension of $K$ that splits $A$, hence its index is at most~$8$
and we may fix a central simple $F(t)$-algebra $A'$ of degree~$8$ such
that
\begin{equation}
  \label{eq:Aprime}
  [A']=[A\otimes_F[a,t)_{F(t)}] \qquad\text{in $\Br(F(t))$.}
\end{equation}

\begin{thm}
  \label{thm:const1}
  The following conditions are equivalent:
  \begin{enumerate}
  \item[(i)]
    $\delta_{K/F}(B)=0$;
  \item[(ii)]
    $A'\simeq Q_1 \otimes_{F(t)}
    Q_2 \otimes_{F(t)}Q_3$ for some
    quaternion $F(t)$-algebras $Q_1$, $Q_2$, $Q_3$;
  \item[(iii)]
    $\Delta(A')=0$.
  \end{enumerate}
  They imply that $A$ decomposes into a tensor product of quaternion
  $F$-algebras. 
\end{thm}

\begin{proof}
    (i)~$\Rightarrow$~(ii): By Proposition~\ref{prop:deltaKFdesc}, (i)
    implies 
  that $B\simeq H_1\otimes_FH_2\otimes_FK$ for some quaternion
  $F$-algebras $H_1$, $H_2$. Since $[B]=[A]_K$, it follows that
  $A\otimes_FH_1\otimes_FH_2$ is split by $K$, hence
  \[
    A\simeq
      H_1\otimes_FH_2\otimes_F[a,u)_F\qquad\text{for some $u\in
        F^\times$}.
  \]
  Then~(ii) holds with
  \[
    Q_1 = H_1\otimes_FF(t),\qquad
    Q_2 = H_2\otimes_FF(t)
    \quad\text{and}\quad
    Q_3 = [a,ut)_{F(t)}.
  \]
  \medbreak

  (ii)~$\Rightarrow$~(iii): This readily follows from
  Proposition~\ref{prop:decomp}.
  \medbreak

  (iii)~$\Rightarrow$~(i):
  Let $\sigma_B$ be a symplectic involution on $B$ and $t_K\colon K\to
  F$ a nonzero $F$-linear functional such that $t_K(1)=0$. 
  The algebra $B\otimes_KK(t)$ is Brauer-equivalent to $A_{K(t)}\simeq
  A'_{K(t)}$, therefore, we may  
  consider $B\otimes_KK(t)$ as a subalgebra of $A'$. Hence Proposition~\ref{prop:sNrp} shows that
  $\Delta(A')$ is represented by
  $(t_K)_*(\Nrp_{\sigma_B})_{F(t)}$. Let us write $\varphi=(t_K)_*(\Nrp_{\sigma_B})$, so that $\varphi$ represents
  $\delta_{K/F}(B)$ and $\varphi_{F(t)}$ represents $\Delta(A')$. We
  assume $\Delta(A')=0$ and we have to show that $\delta_{K/F}(B)=0$.

  Let $q_A\in I^2_q(F)$ be such that $e_2(q_A)=[A]$.
  Then $e_2(q_{A}+\qpf{t,a})=[A']$
  by~\eqref{eq:Aprime}, and $\Delta(A')=0$ means that there exists
  $f\in F(t)^\times$ such that
  \begin{equation}
    \label{eq:const11}
    \varphi_{F(t)} \equiv \bpf f (q_A+\qpf{t,a}) \bmod I^4_q(F(t)).
  \end{equation}
  We use the $t$-adic valuation $v_0$ on $F(t)$.
  Since $\varphi$, $q_A$ and $\qpf{t,a}$ are tame quadratic forms,
  \eqref{eq:const11} is modulo $I^4_q(F(t)) \cap I_q(F(t))_0 =
  I^4_q(F(t))_0$. Multiplying $f$ by a suitable power 
  of $t^2$, we may assume $v_0(f)=0$ or $1$. If $v_0(f)=0$, let
  $f(0)\in F^\times$ denote the residue of $f$. Taking the
  image of each side of~\eqref{eq:const11} under the residue maps
  $\partial_{q1}$ and $\partial_{q+}$
  yields respectively
  \[
    0\equiv \qpf{f(0),a} \bmod I^3_q(F)\qquad\text{and}\qquad
    \varphi\equiv \bpf{f(0)}q_A \bmod I^4_q(F).
  \]
  The first relation shows that $f(0)$ is a norm from $K/F$; the
  second relation then shows that $\varphi$ vanishes in the quotient
  $\overline I^3_q(F)/\overline{\bpf{N_{K/F}(K^\times)}q_A}$, hence
  $\delta_{K/F}(B)=0$. Assume 
  next $v_0(f)=1$ and let $f_0=t^{-1}f$. Using $\bpf{f,t} =
  \bpf{f_0,t}$, we find by applying $\partial_{q1}$ and
  $\partial_{q+}$ to~\eqref{eq:const11}
  \[
    0\equiv \bil{f_0(0)}q_A+\qpf{f_0(0),a} \bmod I^3_q(F)
    \qquad\text{and}\qquad 
    \varphi\equiv \bpf{f_0(0)}q_A \bmod I^4_q(F).
  \]
  Then $\varphi\equiv \qpf{f_0(0),f_0(0),a} \equiv0\bmod I^4_q(F)$, so
  $\overline\varphi=0$ in $\overline I^3_q(F)$. Thus,
  $\delta_{K/F}(B)=0$ in each case, proving that (iii) implies (i).

  To complete the proof, observe that in the proof
  of~(i)~$\Rightarrow$~(ii) above, we also saw that (i) implies that
  $A$ decomposes into a tensor product of three quaternion algebras.
\end{proof}

As a consequence of Theorem~\ref{thm:const1}, for any indecomposable
division $F$-algebra $A$ of degree~$8$ and exponent~$2$, such as those
constructed by Rowen~\cite{Rowen}, McKinnie~\cite{McK} or
Karpenko~\cite{Karp}, the corresponding $F(t)$-algebra $A'$ has
nontrivial $\Delta$-invariant. By Rowen's theorem
in~\cite[Th.~5.6.10]{J}, $A'$ contains a separable triquadratic extension
of $F(t)$, hence also a separable biquadratic extension $L'$ and a
separable quadratic extension $K'\subset L'$. In view
of~\eqref{eq:Gamma}, for $Q'$ and $B'$ the centralizers of $L'$ and $K'$
in $A'$, we have $\delta_{L'/F(t)}(Q')\neq0$ and
$\delta_{K'/F(t)}(B')\neq0$.

\begin{remark*}
  The proof of (iii)~$\Rightarrow$~(i) in Theorem~\ref{thm:const1} is
  directly inspired by the proof of~\cite[Prop.~6.1]{BaCL} by
  Barry--Chapman--Laghribi, which shows in cohomological terms that
  the map
  \[
    \overline I^3_q(F)/\overline{\bpf{N_{K/F}(K^\times)}q_A}
    \longrightarrow 
    \overline I^3_q(F(t))/\overline{\bpf{F(t)^\times}(q_A+\qpf{t,a})}
  \]
  induced by scalar extension is injective. Since this map carries
  $\delta_{K/F}(B)$ to $\Delta(A')$, it readily follows
  that~(iii)~$\Rightarrow$~(i).

  Note, however, that the proof of~\cite[Prop.~6.1]{BaCL} has a gap: it is
  claimed without proof in the first lines of p.~315 that Kato's
  isomorphism induces an isomorphism $H_2^{m+1}(L)_0 \to 
  I_q^{m+1}(L)_0/I_q^{m+2}(L)_0$ (where the $0$-subscript denotes the intersection
  with the tame subgroup), whereas the target group is the quotient of
  $I_q^{m+1}(L)_0$ by its intersection with $I_q^{m+2}(L)$. The claim
  holds, however, because Arason's theorem \cite{Ara2022} yields
  $I_q^{m+2}(L)_0 = I_q^{m+1}(L)_0\cap I_q^{m+2}(L)$. 
\end{remark*}

\subsection*{Second construction}

The following notation is used until the end of the section: $k$ is an
arbitrary field of characteristic~$2$ and $a_1$, $a_2$, $t_1$, $t_2$,
$t_3$ are independent indeterminates over $k$. Let
\[
  F=k(a_1,a_2),\qquad F_1=F(t_1),\qquad F_2=F(t_1,t_2), \qquad
  F_3=F(t_1,t_2,t_3).
\]
Let also $K=F(\sqrt{a_1a_2})$ and $L=K(\wp^{-1}(a_2))$, and
\[
  K_i=F_i(\sqrt{a_1a_2}) \qquad\text{and}\qquad
  L_i=K_i\bigl(\wp^{-1}(a_2)\bigr)\qquad\text{for $i=1$, $2$, $3$.}
\]
Consider the following central simple algebras over $L_1$, $K_2$ and
$F_3$ respectively:
\begin{align}
  \label{eq:A1}
  A_1&=[a_1a_2,a_1+a_2)_F\otimes_F[a_1,t_1)_{F_1}\otimes_{F_1}L_1,
  \\
  \label{eq:A2}
  A_2&=[a_1a_2,a_1+a_2)_F\otimes_F[a_1,t_1)_{F_1}\otimes_{F_1}
       [a_2,t_2)_{F_2}\otimes_{F_2}K_2,
  \\
  \label{eq:A3}
  A_3&=[a_1a_2,a_1+a_2)_F\otimes_F[a_1,t_1)_{F_1}\otimes_{F_1}
       [a_2,t_2)_{F_2}\otimes_{F_2}[t_3,a_1a_2)_{F_3}.
\end{align}
We show in \S\ref{subsec:deltaLFnontriv} that $A_1$ is
Brauer-equivalent to a quaternion $L_1$-algebra $Q$ such that
$\delta_{L_1/F_1}(Q)\neq0$, in \S\ref{subsec:deltaKFnontriv} that
$A_2$ is Brauer-equivalent to a biquaternion $K_2$-algebra $B$ such
that $\delta_{K_2/F_2}(B)\neq0$ and in \S\ref{subsec:Deltanontriv}
that $A_3$ is Brauer-equivalent to a central simple $F_3$-algebra $A$
of degree~$8$ such that $\Delta(A)\neq0$.

\subsection{Nontriviality of the invariant of quaternion algebras}
\label{subsec:deltaLFnontriv}

In order to show that the $L_1$-algebra $A_1$ (defined
in~\eqref{eq:A1}) is 
Brauer-equivalent to a quaternion 
algebra, we use the following observation: for any field $E$ of
characteristic~$2$ and any $\alpha_1$, $\alpha_2\in E$ and $\beta_1$,
$\beta_2\in E^\times$ with $\beta_1\neq\beta_2$, 
\begin{equation}
  \label{eq:observation}
    [\alpha_1,\beta_1)_E\otimes_E [\alpha_2,\beta_2)_E \simeq
    [\alpha_1+\alpha_2, \beta_1+\beta_2)_E \otimes_E \bigl[
    \frac{\alpha_1\beta_2+\alpha_2\beta_1}{\beta_1+\beta_2},
    \beta_1\beta_2\bigr)_E. 
\end{equation}
This can be seen as follows: if $i_1$, $i_2$, $j_1$, $j_2\in
[\alpha_1,\beta_1)_E\otimes_E [\alpha_2,\beta_2)_E$ are such that
\[
  i_r^2+i_r=\alpha_r,\qquad j_r^2=\beta_r, \qquad i_rj_r+j_ri_r=j_r
  \quad\text{for $r=1$, $2$}
\]
and $i_1$, $j_1$ commute with $i_2$ and $j_2$, then $i_1+i_2$ and
$j_1+j_2$ generate a quaternion $E$-algebra isomorphic to $[\alpha_1+\alpha_2,
\beta_1+\beta_2)_E$. Its centralizer contains
$\frac{(i_1j_2+j_1i_2)(j_1+j_2)}{\beta_1+\beta_2}$ and $j_1j_2$,
which generate a quaternion $E$-algebra $\bigl[
\frac{\alpha_1\beta_2+\alpha_2\beta_1}{\beta_1+\beta_2},
\beta_1\beta_2\bigr)_E$.

Applying \eqref{eq:observation} with $E=K$, $\alpha_r=a_r$ and
$\beta_r=a_r+\sqrt{a_1a_2}$ shows that
\[
  [a_1,a_1+\sqrt{a_1a_2})_{K}\otimes_{K}
  [a_2,a_2+\sqrt{a_1a_2})_{K} \simeq
  [a_1+a_2,a_1+a_2)_{K} \otimes_{K} [\sqrt{a_1a_2},
  (a_1+a_2)\sqrt{a_1a_2})_{K}.
\]
The first quaternion factor on the right
side is split; since $a_1a_2\equiv \sqrt{a_1a_2}
\bmod\wp(K)$, the
second is isomorphic to \[[\sqrt{a_1a_2},a_1+a_2)_{K}\simeq [a_1a_2,a_1+a_2)_{K}.\] Therefore,
\begin{equation}
  \label{eq:a1a2}
  \bigl[[a_1a_2,a_1+a_2)_{K}\bigr] =
  \bigl[[a_1,a_1+\sqrt{a_1a_2})_{K}
  \otimes_K[a_2,a_2+\sqrt{a_1a_2})_K\bigr].
\end{equation}
Extending scalars to $L_1$ splits the second factor on the right side
since $a_2\in\wp(L_1)$, hence $[a_1a_2,a_1+a_2)_{L_1} \simeq
[a_1,a_1+\sqrt{a_1a_2})_{L_1}$. 
Substituting in \eqref{eq:A1}, we see that $A_1$ is Brauer-equivalent
to the quaternion algebra
\[
  Q = [a_1,(a_1+\sqrt{a_1a_2})t_1)_{L_1}.
\]
The conditions in Theorem~\ref{thm:biquaddescBr} are satisfied because
$A_1$ has a descent to a tensor product of quaternion $F_1$-algebras,
hence the invariant $\delta_{L_1/F_1}(Q)$ is defined.
We may
use Proposition~\ref{prop:exdeltaLF} to compute
$\delta_{L_1/F_1}(Q)$, since $a_1\in F_1$ and $(a_1+\sqrt{a_1a_2})t_1\in K_1^\times$. Note that \[N_{K_1/F_1}((a_1+\sqrt{a_1a_2})t_1)=a_1(a_1+a_2)t_1^2,\] hence
\[
  \bigl[a_1,N_{K_1/F_1}((a_1+\sqrt{a_1a_2})t_1)\bigr)_{F_1} \simeq
  [a_1,a_1+a_2)_{F_1} \simeq [a_2, a_1+a_2)_{F_1} \simeq
  \bigl[a_2,N_{K_1/F_1}(a_2+\sqrt{a_1a_2})\bigr)_{F_1}.
\]
Therefore, Proposition~\ref{prop:exdeltaLF} with $\alpha_0=a_1$,
$\beta_0=(a_1+\sqrt{a_1a_2})t_1$, $\alpha=a_2$ and
$\beta=a_2+\sqrt{a_1a_2}$, and with $t_{K_1}\colon K_1\to F_1$ given
by $t_{K_1}(1)=0$ and $t_{K_1}(\sqrt{a_1a_2})=1$, shows that
$\delta_{L_1/F_1}(Q)$ is 
represented by $\qpf{t_1(a_1+a_2),a_1a_2,a_1}$. Moreover, the Pfister form $\qpf{a_1+a_2,a_1a_2,a_1}$ represents $a_1^2+a_1a_2+a_1(a_1+a_2)=0$, hence it is hyperbolic. Since in addition $\qpf{a_1a_2,a_1}=\qpf{a_2,a_1}$, we finally get that $\delta_{L_1/F_1}(Q)$ is represented by 
$\qpf{t_1,a_2,a_1}$.

\begin{thm}
  \label{thm:deltaLFnontriv}
  $\delta_{L_1/F_1}(Q)\neq0$.
\end{thm}

\begin{proof}
  In view of the calculation of $\delta_{L_1/F_1}(Q)$ above, we have
  to show
  \[
    \overline{\qpf{t_1,a_2,a_1}} \notin
    \overline{(s_{L_1})_*(\bpf{L_1^\times}\Nrd_Q)}
  \]
  for $s_{L_1}\colon L_1\to F_1$ an arbitrary nonzero $F_1$-linear
  functional. 
  Since $Q$ is Brauer-equivalent to $A_1$, it follows from the
  definition of $A_1$ that
  \[
    \overline{\Nrd_Q} = \overline{\bigl(\qpf{a_1+a_2,a_1a_2} +
      \qpf{t_1,a_1}\bigr)_{L_1}} \qquad\text{in $\overline
      I_q^2(L_1)$}.
  \]
  Therefore, by Frobenius reciprocity
  \[
    \overline{(s_{L_1})_*(\bpf{L_1^\times}\Nrd_Q)} =
    \overline{\bpf{N_{L_1/F_1}(L_1^\times)}\bigl(\qpf{a_1+a_2,a_1a_2}
      + \qpf{t_1,a_1}\bigr)} \qquad\text{in $\overline I^3_q(F_1)$}.
  \]
  By way of contradiction, suppose there exists $\ell\in L_1^\times$
  such that
  \begin{equation}
    \label{eq:deltaLFnontriv1}
    \qpf{t_1,a_2,a_1} \equiv \bpf{N_{L_1/F_1}(\ell)}
    \bigl(\qpf{a_1+a_2,a_1a_2} + \qpf{t_1,a_1}\bigr) \bmod I_q^4(F_1).
  \end{equation}
  Consider the $t_1$-adic valuation $v_0$ on $L_1$ (with constant
  field $L$) 
  and let $\ell=\ell_0t_1^n$ with $n=v_0(\ell)$, so $\ell_0$ is a unit
  for the $v_0$-valuation. Its residue is $\ell_0(0)\in
  L^\times$ and the residue of $N_{L_1/F_1}(\ell_0)$ is
  $N_{L/F}(\ell_0(0))$. 
  Since $a_1$, $a_2\in F$, all the
  terms in~\eqref{eq:deltaLFnontriv1} are in the tame subgroup
  $I_q(F_1)_0$, hence the equation is modulo $I_q^4(F_1)\cap
  I_q(F_1)_0 = I_q^4(F_1)_0$. Since $N_{L_1/F_1}(\ell) =
  N_{L_1/F_1}(\ell_0) t_1^{4n}$, we have $\bpf{N_{L_1/F_1}(\ell)}
  \simeq \bpf{N_{L_1/F_1}(\ell_0)}$, hence we may substitute
  $N_{L_1/F_1}(\ell_0)$ for $N_{L_1/F_1}(\ell)$
  in~\eqref{eq:deltaLFnontriv1}. Applying the residue map
  $\partial_{q1}$, we then obtain
  \[
    \qpf{a_2,a_1} \equiv \qpf{N_{L/F}(\ell_0(0)), a_1} \bmod I^3_q(F),
  \]
  hence $\qpf{a_2N_{L/F}(\ell_0(0)), a_1}$ is hyperbolic, and therefore
  \begin{equation}
    \label{eq:deltaLFnontriv2}
    a_2N_{L/F}(\ell_0(0)) = N_{F(\wp^{-1}(a_1))/F}(u) \qquad\text{for
      some $u\in F(\wp^{-1}(a_1))^\times$.}
  \end{equation}
  To complete the proof, we show that \eqref{eq:deltaLFnontriv2} is
  impossible. For this, we use the $a_2^{-1}$-adic valuation $v$ on
  $F=k(a_1)(a_2^{-1})$ (with constant field $k(a_1)$). This valuation
  extends uniquely and without ramification to $F(\wp^{-1}(a_1))$ with
  residue field $k(a_1)(\wp^{-1}(a_1))$, hence, writing again $v$ for
  the unique extension of $v$ to $F(\wp^{-1}(a_1))$,
  \[
    v(N_{F(\wp^{-1}(a_1))/F}(u)) = 2v(u)\in2\Z.
  \]
  To see that $v$ also extends uniquely to $L$, observe that, writing
  $w$ for an extension of $v$ to $L$, we must have
  \[
    w(\sqrt{a_1a_2}) = \textstyle{\frac12}v(a_2)=-\textstyle{\frac12}
    \quad\text{and also}\quad
    w(\wp^{-1}(a_2)) = \textstyle{\frac12}v(a_2)=-\textstyle{\frac12}
  \]
  because $\wp^{-1}(a_2)^2+\wp^{-1}(a_2)=a_2$. Moreover,
  \[
    (\wp^{-1}(a_2)\sqrt{a_1a_2}^{-1})^2 = \wp^{-1}(a_2)(a_1a_2)^{-1} +
    a_1^{-1}
  \]
  with $w(\wp^{-1}(a_2)(a_1a_2)^{-1})=\frac12$, hence the residue
  field $\overline L$ contains $\sqrt{a_1}$. From the fundamental
  inequality
  \[
    [L{:}F]\geq [\overline L{:}\overline F]\,
    (w(L^\times){:}w(F^\times))
  \]
  it then follows that
  $w(L^\times)=\frac12\Z$ and $\overline L = k(\sqrt{a_1})$, hence $w$
  is the unique extension of $v$ to $L$. Therefore,
  \[
    v(N_{L/F}(\ell_0(0))) = 4w(\ell_0(0)) \in 2\Z.
  \]
  Equation~\eqref{eq:deltaLFnontriv2} then yields $v(a_2)\in2\Z$, a
  contradiction. 
\end{proof}

\subsection{Nontriviality of the invariant of biquaternion algebras}
\label{subsec:deltaKFnontriv}

Using~\eqref{eq:a1a2} to substitute $[a_1a_2,a_1+a_2)_K$
in~\eqref{eq:A2} we see that $A_2$ is Brauer-equivalent to the
following biquaternion $K_2$-algebra:
\[
  B=[a_1,(a_1+\sqrt{a_1a_2})t_1)_{K_2}\otimes_{K_2}
  [a_2,(a_2+\sqrt{a_1a_2})t_2)_{K_2}.
\]
Theorem~\ref{thm:quadesc} shows that the condition $N_{K_2/F_2}[B]=0$
is satisfied, because $A_2$ has a descent to a tensor product of
quaternion $F_2$-algebras. We may therefore consider the invariant
$\delta_{K_2/F_2}(B)$.

Since $L_2$ splits the second factor in the definition of $B$, we have
$[B]_{L_2} = [Q]_{L_2}$, 
hence we may identify $Q_{K_2}$ with the centralizer of $L_2$ in
$B$. Therefore, $\delta_{K_2/F_2}(B) =
\Phi(\delta_{L_2/F_2}(Q_{L_2}))$, and the computation of
$\delta_{L_1/F_1}(Q)$ in \S\ref{subsec:deltaLFnontriv} shows that
$\delta_{K_2/F_2}(B)$ is represented by the form $\qpf{t_1,a_2,a_1}$.

\begin{thm}
  \label{thm:deltaKFnontriv}
  $\delta_{K_2/F_2}(B)\neq0$.
\end{thm}

\begin{proof}
  Since $B$ is Brauer-equivalent to $A_2$, it follows from the
  definition of $A_2$ in~\eqref{eq:A2} that a form $q_B\in I_q^2(K_2)$
  such that $e_2(q_B) = [B]$ is given by
  \[
    q_B = (\qpf{a_1+a_2,a_1a_2} + \qpf{t_1,a_1} + \qpf{t_2,a_2})_{K_2}.
  \]
  By Frobenius reciprocity, we have
  \[
    \overline{(s_{K_2})_*(\bpf{K_2^\times}q_B)} =
    \overline{\bpf{N_{K_2/F_2}(K_2^\times)} (\qpf{a_1+a_2,a_1a_2} +
      \qpf{t_1,a_1} + \qpf{t_2,a_2})}.
  \]
  If $\delta_{K_2/F_2}(B)=0$, then there exists $g\in K_2^\times$ such
  that
  \begin{equation}
    \label{eq:deltaKFnontriv1}
    \qpf{t_1,a_2,a_1} \equiv \bpf{N_{K_2/F_2}(g)} \bigl(
    \qpf{a_1+a_2,a_1a_2} +\qpf{t_1,a_1} + \qpf{t_2,a_2}\bigr) \bmod
    I^4_q(F_2). 
  \end{equation}
  Our goal is to show that this equation leads to a contradiction.
  
  Consider the $t_2$-adic valuation $v_0$ on $F_2$ with constant field
  $F_1$. Multiplying $g$ by a suitable power of $t_2$, which does not
  change the square class of $N_{K_2/F_2}(g)$, we may assume that
  $v_0(g)=0$, which means that $g(0)\neq0$ when $g$ is 
  viewed in $F_2=F_1(t_2)$. Since $a_1$, $a_2$ and $a_1a_2$ are
  constants, all the terms in~\eqref{eq:deltaKFnontriv1} are in the
  tame subgroup $I_q(F_2)_0$, hence the equation is
  modulo~$I^4_q(F_2)\cap I_q(F_2)_0= I^4_q(F_2)_0$. Applying to each
  side of~\eqref{eq:deltaKFnontriv1} the residue map $\partial_{q1}$,
  we obtain
  \[
    0\equiv \qpf{N_{K_1/F_1}(g(0)),a_2} \bmod I^3_q(F_1),
  \]
  hence the quaternion algebra $[a_2,N_{K_1/F_1}(g(0)))_{F_1}$ is
  split. By Proposition~\ref{prop:descquat}, we may find $c\in
  F_1^\times$ such that $[a_2,g(0))_{K_1}\simeq
  [a_2,c)_{F_1}\otimes_{F_1}K_1$. The quaternion algebra
  $[a_2,g(0)c^{-1})_{K_1}$ is then split, so there exists $\ell\in
  L_1^\times$ such that $g(0)c^{-1}=N_{L_1/K_1}(\ell)$. It follows
  that $N_{K_1/F_1}(g(0)) = c^2N_{L_1/F_1}(\ell)$, hence
  \begin{equation}
    \label{eq:deltaKFnontriv2}
    N_{K_1/F_1}(g(0)) \equiv N_{L_1/F_1}(\ell) \bmod F_1^{\times2}.
  \end{equation}
  On the other hand, applying $\partial_{q+}$ to each side of
  \eqref{eq:deltaKFnontriv1} yields
  \[
    \qpf{t_1,a_2,a_1} \equiv \bpf{N_{K_1/F_1}(g(0))} \bigl(
    \qpf{a_1+a_2,a_1a_2} +\qpf{t_1,a_1} \bigr) \bmod
    I^4_q(F_1).
  \]
  In view of \eqref{eq:deltaKFnontriv2}, we may substitute
  $N_{L_1/F_1}(\ell)$ for $N_{K_1/F_1}(g(0))$; we then see
  that~\eqref{eq:deltaLFnontriv1} holds, a contradiction to
  Theorem~\ref{thm:deltaLFnontriv}. 
\end{proof}

\begin{remark}
  \label{rem:deltaKFtriv}
  Consider the biquaternion $K_1$-algebra
  \[
    B'=[a_1a_2,a_1+a_2)_{F_1}\otimes_{F_1} [a_1,t_1)_{F_1}
    \otimes_{F_1}K_1.
  \]
  Using~\eqref{eq:a1a2}, we see that
  \[
    B'\simeq [a_1,(a_1+\sqrt{a_1a_2})t_1)_{K_1} \otimes_{K_1}
    [a_2,a_2+\sqrt{a_1a_2})_{K_1},
  \]
  hence $[B']_{L_1} = [Q]$. We may therefore identify $Q$ with the
  centralizer of $L_1$ in $B'$, and from~\eqref{eq:Gamma} we obtain
  \[
    \Phi\bigl(\delta_{L_1/F_1}(Q)\bigr) =
    \delta_{K_1/F_1}(B').
  \]
  Now, $\delta_{L_1/F_1}(Q)\neq0$ by Theorem~\ref{thm:deltaLFnontriv},
  but $\delta_{K_1/F_1}(B')=0$ by Proposition~\ref{prop:deltaKFdesc}
  since $B'$ has a descent to $F_1$.
\end{remark}

\subsection{Nontriviality of the $\Delta$-invariant}
\label{subsec:Deltanontriv}

From~\eqref{eq:a1a2}, it follows that all the quaternion algebras in
the definition of $A_3$ in~\eqref{eq:A3} are split by the triquadratic
extension $F_3(\wp^{-1}(a_1),\wp^{-1}(a_2),\sqrt{a_1a_2})$ of
$F_3$. Therefore, $A_3$ is Brauer-equivalent to a central simple
$F_3$-algebra $A$ of degree~$8$, and the invariant $\Delta(A)$ is
defined. 
A quadratic form $q_A\in I^2_q(F_3)$ such that $e_2(q_A)=[A]$ is 
given by
\begin{equation}
  \label{eq:qA}
  q_A=\qpf{a_1+a_2,a_1a_2} + \qpf{t_1,a_1} + \qpf{t_2,a_2} +
  \qpf{a_1a_2,t_3}.
\end{equation}
Extending scalars from $F_3$ to $K_3$
splits the last quaternion factor on the right side of~\eqref{eq:A3},
hence
\[
  [A]_{K_3} = [A_3]_{K_3} = [A_2]_{K_3} = [B]_{K_3}.
\]
Therefore, we may identify $B_{K_3}$ with the centralizer of $K_3$ in
$A$.
It follows that $\Delta(A) = \Psi(\delta_{K_3/F_3}(B_{K_3}))$,
and the computation of $\delta_{K_2/F_2}(B)$
in~\S\ref{subsec:deltaKFnontriv} shows that $\Delta(A)$ is
represented by $\qpf{t_1,a_2,a_1}$.
\medbreak

The proof below that $\Delta(A)\neq0$ follows the same pattern as the
proofs of Theorem~\ref{thm:deltaLFnontriv} and
Theorem~\ref{thm:deltaKFnontriv}, using the $t_3^{-1}$-adic valuation
$v_\infty$ on $F_3=F_2(t_3)$ with constant field $F_2$. For $f\in F_3$
with $v_\infty(f)\geq0$ we write $f(\infty)$ for the image of $f$ in
the residue field $F_2$. In contrast
with~\S\ref{subsec:deltaLFnontriv} and \S\ref{subsec:deltaKFnontriv},
the use of residue maps is more delicate because the form $q_A$ does
not belong to the tame subgroup $I_q(F_3)_0$. Since
$v_\infty(1)+v_\infty(t_3)=-1$, the form $[1,t_3]$ lies in the 
first level $I_q(F_3)_1$ of Arason's filtration
\cite[Th.~3]{Ara2018}, hence also $\qpf{a_1a_2,t_3}\in I_q(F_3)_1$
because $I_q(F_3)_1$ is a module over $W(F_3)$. To prove that
$\Delta(A)\neq0$, we need the following lemma:

\begin{lemma}
  \label{lem:tech3}
  If $f\in F_3^\times$ is such that $\qpf{f,a_1a_2,t_3}
  \in I^3_q(F_3)_0+I^4_q(F_3)$, then there exists
  $\psi\in I^3_q(F_3)_0$ such that
  \[
    \qpf{f,a_1a_2,t_3}\equiv \psi \bmod I_q^4(F_3)
    \qquad\text{and}\qquad
    \partial_{q+}(\psi)=0.
  \]
  Moreover, if $v_\infty(f)=n$, then  $f_0(\infty)\in
  N_{K_2/F_2}(K_2^\times)$ for $f_0=ft_3^n$.
\end{lemma}

\begin{proof}
  Since $\qpf{f,t_3} = \qpf{ft_3^n,t_3}$, we may
  substitute $f_0$ for $f$ and we are thus reduced to
  consider the case where $v_\infty(f)=0$. We assume this for the rest
  of the proof.
  
  Since by definition the module $I_q(F_3)_1$ contains
  $I_q(F_3)_0$, the hypothesis implies that $\qpf{f,a_1a_2,t_3}$
  lies in $I^3_q(F_3)_0+(I^4_q(F_3)\cap
  I_q(F_3)_1)$. In~\cite{Ara2022}, Arason shows that 
  $I^4_q(F_3)\cap I_q(F_3)_1=I^4_q(F_3)_1:=I^3(F_3)\cdot
  I_q(F_3)_1$. Moreover, he 
  describes on~\cite[p.~79]{Ara2022} a homomorphism from
  $I^3_q(F_3)_1$ to the vector space $\Omega^2_{F_2}$ of absolute
  $2$-differential forms over $F_2$. This homomorphism vanishes on
  $I^3_q(F_3)_0$ and on 
  $I^4_q(F_3)_1$, and maps $\qpf{f,a_1a_2,t_3}$ to $\frac{d
    f(\infty)}{f(\infty)}\wedge \frac{d(a_1a_2)}{a_1a_2}$, hence
  \[
    \frac{d f(\infty)}{f(\infty)}\wedge \frac{d(a_1a_2)}{a_1a_2}=0
    \qquad\text{in $\Omega^2_{F_2}$.}
  \]
  This equation does not hold if $1$, $f(\infty)$, $a_1a_2$ and
  $f(\infty) a_1a_2$ are $F_2^2$-linearly independent in $F_2$, for then
  $\frac{df(\infty)}{f(\infty)}$ and $\frac{d(a_1a_2)}{a_1a_2}$ are part
  of an $F_2$-base of $\Omega^1_{F_2}$. Therefore, there exist $x$,
  $y\in F_2$ such that
  \[
    f(\infty)=x^2+a_1a_2y^2 =N_{K_2/F_2}(x+y\sqrt{a_1a_2}).
  \]
  To complete
  the proof, we construct a tame form $\psi$ with $\psi\equiv
  \qpf{f,a_1a_2,t_3} \bmod I^4_q(F_3)$.

  Let $s=t_3(f+x^2+a_1a_2y^2)\in F_3$, so $v_\infty(s)\geq0$. Then
  \[
    t_3f^{-1}=(x^2+a_1a_2y^2)^{-1}(t_3+f^{-1}s).
  \]
  Since $\qpf{f,t_3} = \qpf{t_3f^{-1},t_3}$, it follows that
  \[
    \qpf{f,t_3} \equiv \qpf{x^2+a_1a_2y^2,t_3} + \qpf{t_3+f^{-1}s,t_3}
    \bmod  I^3_q(F_3).
  \]
  Multiplying by $\bpf{a_1a_2}$ and taking into account that
  $\bil{x^2+a_1a_2y^2}\bpf{a_1a_2} = \bpf{a_1a_2}$, we obtain
  \[
    \qpf{f,a_1a_2,t_3}\equiv \qpf{a_1a_2, t_3+f^{-1}s,t_3} \bmod
    I^4_q(F_3). 
  \]
  Now, use $\qpf{t_3+f^{-1}s,t_3+f^{-1}s}=0$ to rewrite the right side:
  \[
    \qpf{f,a_1a_2,t_3} \equiv \qpf{a_1a_2,t_3+f^{-1}s,f^{-1}s} \bmod
    I^4_q(F_3).
  \]
  Letting $\psi=\qpf{a_1a_2,t_3+f^{-1}s,f^{-1}s} = \qpf{a_1a_2,
    t_3(1+f^{-1}st_3^{-1}), f^{-1}s}$, we have $\psi\in I^3_q(F_3)_0$
  and $\partial_{q+}(\psi)=0$ because $(f^{-1}st_3^{-1})(\infty)=0$. 
\end{proof}

\begin{thm}
  \label{thm:Deltanontriv}
  $\Delta(A)\neq0$.
\end{thm}

\begin{proof}
  As observed at the beginning of this subsection, $\Delta(A)$ is
  represented by $\qpf{t_1,a_2,a_1}$. Therefore, $\Delta(A)=0$ means
  that there exists $f\in F_3^\times$ such that
  \begin{equation}
    \label{eq:Deltanontriv1}
    \qpf{t_1,a_2,a_1} \equiv \bpf{f}q_A \bmod I^4_q(F_3).
  \end{equation}
  In view of the choice of $q_A$
  in~\eqref{eq:qA}, this equation implies
  \[
    \qpf{f,a_1a_2,t_3} \equiv \qpf{t_1,a_2,a_1} - \qpf{f,a_1+a_2,a_1a_2}
    -\qpf{f,t_1,a_1} - \qpf{f,t_2,a_2} \bmod I^4_q(F_3).
  \]
  Since all the quadratic forms on the right side are in
  $I_q(F_3)_0\cap I_q^3(F_3) = I^3_q(F_3)_0$, it follows that
  $\qpf{f,a_1a_2,t_3}$ belongs to
  $I^3_q(F_3)_0+I^4_q(F_3)$. Lemma~\ref{lem:tech3} then yields
  $\psi\in I_q^3(F_3)_0$ such that $\qpf{f,a_1a_2,t_3} \equiv
  \psi\bmod I^4_q(F_3)$ and $\partial_{q+}(\psi)=0$. Substituting
  in~\eqref{eq:Deltanontriv1}, we obtain
  \[
    \qpf{t_1,a_2,a_1} \equiv \bpf{f}\bigl(\qpf{a_1+a_2,a_1a_2} +
    \qpf{t_1,a_1} + \qpf{t_2,a_2}\bigr) + \psi \bmod I^4_q(F_3).
  \]
  Letting $f_0=ft_3^{v_\infty(f)}$ and applying $\partial_{q+}$ to
  each side of the last displayed equation yields 
  \begin{equation}
    \label{eq:Deltanontriv2}
    \qpf{t_1,a_2,a_1} \equiv
    \bpf{f_0(\infty)}\bigl(\qpf{a_1+a_2,a_1a_2} + \qpf{t_1,a_1} +
    \qpf{t_2,a_2}\bigr) \bmod I^4_q(F_2),
  \end{equation}
  But Lemma~\ref{lem:tech3} also
  yields $f_0(\infty)= N_{K_2/F_2}(g)$ for some $g\in K_2^\times$,
  so~\eqref{eq:Deltanontriv2} shows that \eqref{eq:deltaKFnontriv1}
  holds, contradicting
  Theorem~\ref{thm:deltaKFnontriv}. Thus, there is no $f\in
  F_3^\times$ such that \eqref{eq:Deltanontriv1} holds, hence
  $\Delta(A)\neq0$. 
\end{proof}

Note that $A$ is split by
$F_3(\wp^{-1}(a_1), \wp^{-1}(a_2), \sqrt{a_1a_2})$, and since
$\Delta(A)\neq0$ it does not decompose into a tensor product of
quaternion algebras.
Therefore it provides a counterexample to~\cite[Th.~4.1(iii)]{AJOR},
hence also to the ``kernel splitting property'' of the extension
$F_3\bigl(\wp^{-1}(a_1), \wp^{-1}(a_2)\bigr)\cdot
F_3(\sqrt{a_1a_2})/F_3$ asserted
in~\cite[Th.~1.2]{ALOR}, which is the sole ingredient in the proof
of~\cite[Th.~4.1(iii)]{AJOR}. 

\begin{remark}
  \label{rem:Deltatriv}
  Consider the $F_2$-algebra
  \[
    A'=[a_1a_2,a_1+a_2)_{F_2} \otimes_{F_2} [a_1,a_2t_1)_{F_2}
    \otimes_{F_2} [a_2,t_2)_{F_2}.
  \]
  Extending scalars to $K_2$ splits $[a_1,a_2)_{F_2}$ since
  $[a_1,a_2)_{K_2}\simeq [a_1,a_1)_{K_2}$, hence $[A']_{K_2}=[B]$ and
  we may identify $B$ with the centralizer of $K_2$ in $A'$. Then
  \eqref{eq:Gamma} yields
  \[
    \Psi_{K_2/F_2}\bigl(\delta_{K_2/F_2}(B)\bigr) = \Delta(A').
  \]
  Now, $\delta_{K_2/F_2}(B)\neq0$ by Theorem~\ref{thm:deltaKFnontriv},
  but $\Delta(A')=0$ since $A'$ is a tensor product of three
  quaternion algebras.
\end{remark}

\section*{Appendix: Quaternion algebras over purely inseparable
  biquadratic extensions}
\setcounter{section}{1}
\setcounter{prop}{0}
\renewcommand{\thesection}{\Alph{section}}

Purely inseparable biquadratic field extensions are excluded in
Section~\ref{sec:biquadesc}. This appendix highlights how this case
differs from those considered above. No invariant is needed here as a quaternion algebra always descends when its Brauer class does. More precisely, we have the following: 

\begin{thm}
  \label{thm:pureinsep}
  Let $L$ be a purely inseparable biquadratic field extension (i.e., a
  purely inseparable quartic extension of exponent one) of an
  arbitrary field $F$. For a quaternion $L$-algebra $Q$, the following
  are equivalent:
  \begin{enumerate}
  \item[(i)]
    there exists a central simple $F$-algebra $A$ such that $2[A]=0$
    and $[A]_L=[Q]$;
  \item[(ii)]
    there exists a quaternion $F$-algebra $Q_0$ such that $Q\simeq
    Q_0\otimes_FL$. 
  \end{enumerate}
  Moreover, when these conditions hold, for every central simple
  $F$-algebra $A$ of degree~$8$ and exponent~$2$ containing $Q$, the following hold:
  \begin{enumerate}
  \item[(a)]
    $A\simeq Q_0\otimes_FQ_1\otimes_FQ_2$ for some quaternion
    $F$-algebras $Q_0$, $Q_1$, $Q_2$ such that $L\subset
    Q_1\otimes_FQ_2$;
  \item[(b)]
    $\pi(\sigma)$ is hyperbolic for every involution $\sigma$ on $A$
    that extends the conjugation on $Q$.
  \end{enumerate}
\end{thm}

The proof uses the following elementary ring-theoretic observation:

\begin{lemma}
  \label{lem:pureinsep}
  Let $E$ be a central simple algebra over a field $Z$ of
  characteristic~$2$ and let $\theta$ be an involution (of any type)
  on $E$. Let also $Z(r)\subset E$ be an inseparable quadratic
  extension of $Z$. If $r=u+\theta(u)$ for some $u\in E$ that commutes
  with $r$, then there exists $e\in\Symd(\theta)$ such that $r=re+er$.
\end{lemma}

\begin{proof}
  Consider the $Z$-linear map $r*\text{\textvisiblespace}\colon E\to
  E$ that maps $x\in E$ to $rx+xr$. Its kernel is the centralizer
  $C$ of $r$ in $E$, and $(r*\text{\textvisiblespace})^2=0$ because
  $r^2\in Z^\times$. Therefore, the image of
  $r*\text{\textvisiblespace}$ lies in $C$. But
  $\dim_ZC=\frac12\dim_ZE$ since $Z$ is a field, hence dimension count
  shows that $C$ is also the image of $r*\text{\textvisiblespace}$. In
  particular, $u=rx+xr$ for some $x\in E$. Then
  $\theta(u)=\theta(x)r+r\theta(x)$, hence
  \[
    r=u+\theta(u)=r\bigl(x+\theta(x)\bigr) + \bigl(x+\theta(x)\bigr)r.
  \]
  This yields $r=re+er$ with $e=x+\theta(x)\in\Symd(\theta)$.
\end{proof}

\begin{proof}[Proof of Theorem~\ref{thm:pureinsep}]
  (ii)~$\Rightarrow$~(i) If (ii) holds, then (i) holds with $A=Q_0$.

  (i)~$\Rightarrow$~(ii) If $A$ satisfies~(i), then $A$ is split by
  the quadratic extensions of $L$ that split $Q$, hence up to
  Brauer-equivalence we may assume that $\deg A=8$ and $A$ contains
  $Q$. 

  Let $\sigma$ be an involution on $A$ that extends the conjugation
  $\gamma$ on $Q$. We show below that~(a) holds with quaternion
  $F$-algebras $Q_0$, $Q_1$, $Q_2\subset A$ stable under
  $\sigma$. Besides proving~(a), this proves~(ii) since $Q$ is the
  centralizer of $L$ in $A$, and also~(b) by the main theorem
  in~\cite{BGBT3} or by~\cite[Prop.~8]{CRAS}.

  Pick $r_1$, $r_2\in L$ such that $L=F(r_1,r_2)$ and let $B\subset A$
  be the centralizer of $r_2$. Thus,
  \[
    F\subset F(r_2)\subset L\subset Q\subset B\subset A.
  \]
  Let also $\sigma_B$ be the restriction of $\sigma$ to $B$, so
  $\sigma_B$ also restricts to $\gamma$ on $Q$. Since
  $L=\Symd(\gamma)$, there exists $u_1\in Q$ such that
  $r_1=u_1+\sigma_B(u_1)$. Lemma~\ref{lem:pureinsep} yields
  $e_1\in\Symd(\sigma_B)$ such that
  \[
    r_1=r_1e_1+e_1r_1,\qquad\text{hence}\quad
    r_1e_1r_1^{-1}=e_1+1.
  \]
  Then $r_1e_1^{2^k}r_1^{-1} = e_1^{2^k}+1$ for all $k\geq0$; in
  particular $e_1^2+e_1$ commutes with $r_1$. Since $e_1\in B$, $e_1^2+e_1$ also
  commutes with $r_2$, and it follows that $e_1^2+e_1\in
  Q$. Moreover, $e_1^2+e_1\in\Sym(\gamma)$ because
  $e_1\in\Symd(\sigma_B)$. Therefore, $(e_1^2+e_1)^2\in L$, hence
  $(e_1^2+e_1)^4\in F$. Set $i_1=e_1^4$. Then $i_1^2+i_1\in F$ and
  $r_1i_1r_1^{-1}=i_1+1$, hence $i_1$ and $r_1$ generate a quaternion
  $F$-algebra $Q_1\subset A$, which is stable under $\sigma$ since
  $i_1$, $r_1\in\Sym(\sigma)$.

  Let $A'\subset A$ be the centralizer of $Q_1$, and write $\sigma_1$,
  $\sigma'$ for the restrictions of $\sigma$ to $Q_1$ and $A'$
  respectively. Under the map $Q_1\times A'\to A$ defined by
  multiplication in $A$, we thus have
  \begin{equation}
    \label{eq:pureinsep}
    (Q_1,\sigma_1)\otimes_F(A',\sigma') \simeq (A,\sigma).
  \end{equation}
  The involution $\sigma_1$ is orthogonal because
  $i_1\in\Sym(\sigma_1)$ while the reduced trace of $i_1$ in $Q_1$ is
  $1$. Therefore, 
  by~\cite[(2.23)]{BoI} the involution $\sigma'$ is symplectic.
   
  Note that $r_2$ commutes with $r_1$, but also with $i_1$, since the latter is in $B$; therefore, $r_2\in A'$.
  Just as for $r_1$ above, we
  may find $u_2\in Q$ such that
  $r_2=u_2+\sigma(u_2)$. Using~\eqref{eq:pureinsep} as an
  identification, write
  \[
    u_2=1\otimes a_0+r_1\otimes a_1+i_1\otimes a_2 + i_1r_1\otimes a_3
    \qquad\text{with $a_0,a_1,a_2,a_3\in A'$.}
  \]
  Since $u_2$ commutes with $r_1$, we have $a_2=a_3=0$. Moreover,
  $a_0$ and $a_1$ commute with $r_2$ because $u_2\in Q$ commutes with
  $r_2\in L$. Now,
  \[
    r_2=1\otimes r_2 = u_2+\sigma(u_2) = 1\otimes a_0 + r_1\otimes a_1
    + 1\otimes\sigma'(a_0) +r_1\otimes\sigma'(a_1),
  \]
  hence $r_2=a_0+\sigma'(a_0)$. Lemma~\ref{lem:pureinsep} then yields
  $e_2\in\Symd(\sigma')$ such that
  \[
    r_2=r_2e_2+e_2r_2.
  \]
  As above, $e_2^2+e_2$ commutes with $r_2$; since $r_2\not=0$, $e_2$ does not commute with $r_2$ and it follows that $F(e_2^2+e_2)\subsetneq
  F(e_2)$. But since $\sigma'$ is a symplectic involution on a
  biquaternion algebra, the minimal polynomial of $e_2$ over $F$ has
  degree at most~$2$, hence $e_2^2+e_2\in F$. Therefore, $e_2$ and
  $r_2$ generate a quaternion $F$-algebra $Q_2\subset A'$, which is
  stable under $\sigma'$ since $e_2$, $r_2\in\Sym(\sigma')$. Let
  $Q_0\subset A'$ be the centralizer of $Q_2$, and write $\sigma_2$
  and $\sigma_0$ for the restrictions of $\sigma'$ to $Q_2$ and $Q_0$
  respectively. Then from~\eqref{eq:pureinsep} we obtain
  \[
    (A,\sigma) \simeq (Q_0,\sigma_0) \otimes_F(Q_1,\sigma_1)
    \otimes_F(Q_2,\sigma_2)
  \]
  with $r_1\in Q_1$ and $r_2\in Q_2$, hence $L\subset
  Q_1\otimes_FQ_2$. 
\end{proof}

\begin{remark}
\label{rem:cex}
  Aravire--Jacob--O'Ryan give in \cite[Th.~4.1(ii)]{AJOR} a
  cohomological proof of the following statement:
  \begin{quote}
    \emph{If a central division $F$-algebra $D$ of degree~$8$ and
      exponent~$2$ is split by a field extension of the form
      $F(\wp^{-1}(a),\sqrt x,\sqrt y)$, then $D\simeq
      [a,t)_F\otimes_F[b_1,x)_F\otimes_F[b_2,y)_F$ for some $t\in
      F^\times$ and $b_1$, $b_2\in F$.}
  \end{quote}
  A significant part of Theorem~\ref{thm:pureinsep} follows from this statement. Indeed, if $Q/L$ is a quaternion algebra whose Brauer class descends, then $Q$ is contained in a central simple algebra $A$ of degree $8$ and exponent $2$, to which we may apply the above statement. The decomposition then provides a descent for $Q$.  
  Note that \cite[Th.~4.1(ii)]{AJOR} also claims that every central
  division $F$-algebra of degree~$8$ and exponent~$2$ split by a field
  extension of the form $F(\wp^{-1}(a),\sqrt[4]x)$ decomposes as a
  tensor product of three quaternion algebras. A counterexample to the
  latter statement is obtained as follows:

  Let $M/F$ be a cyclic field extension of degree~$8$; write $\alpha$
  for a generator of its Galois group and consider a cyclic algebra
  \[
    C=\bigoplus_{i=0}^7Mz^i \qquad\text{where $z^8\in F^\times$ and
      $zm=\alpha(m)z$ for all $m\in M$.}
  \]
  Let $K\subset M$ be the unique subfield that is a quadratic
  extension of $F$, i.e., the subfield fixed under $\alpha^2$. If $C$
  is a division algebra, then $K(z^2)$ is a 
  maximal subfield of $C$ of the form
  $F(\wp^{-1}(a),\sqrt[4]x)$. Twisted Laurent series constructions
  easily yield such examples where $C$ has exponent~$8$. Extending
  scalars to the function field $F_1$ of the Severi--Brauer variety of
  $C\otimes_FC$ then yields a division algebra $C_{F_1}$ that is
  `generic' in Karpenko's sense \cite[Ex.~4.12]{Karp}. This algebra is
  therefore indecomposable by~\cite[Cor.~5.4]{Karp}, even though it is
  split by $F_1\cdot K(z^2)$, which is of the form
  $F_1(\wp^{-1}(a),\sqrt[4]x)$.

  Incidentally, this example also shows that
  Theorem~\ref{thm:pureinsep} does not hold for purely
  inseparable quartic field extensions $L/F$ of exponent two instead
  of one.
\end{remark}

\bibliographystyle{plain}
\bibliography{QT_Delta}
\end{document}